\documentclass[12pt,a4paper]{article}

\usepackage{amsfonts,amssymb,amsthm}
\usepackage{amsbsy}
\usepackage{graphicx}
\usepackage{latexsym}
\usepackage{amsmath,enumerate,amsfonts,amssymb,color,graphicx,amsthm}

\usepackage{CJK,CJKnumb}

\def\Bbb{\mathbb}

\theoremstyle{plain}
\newtheorem{theorem}{Theorem}[section]
\newtheorem{corollary}[theorem]{Corollary}
\newtheorem{lemma}[theorem]{Lemma}
\newtheorem{prop}[theorem]{Proposition}
\theoremstyle{definition}
\newtheorem{definition}{Definition}[section]
\theoremstyle{remark}
\newtheorem{remark}{Remark}[section]
\numberwithin{equation}{section}

\theoremstyle{example}

\newcounter{marnote}

\def\o{\overline}
\def\u{\underline}

\def\disp{\displaystyle} \def\s{\setlength\arraycolsep{2pt}}
\def\b{\backslash}\def\dint{\displaystyle\int}

\begin{document}

\begin{center}{\Large The exterior Dirichlet problems of Hessian quotient equations}
\end{center}
\centerline{\large  Limei Dai \quad Jiguang Bao \quad Bo Wang }
\vspace{5mm}

\begin{minipage}{140mm}{\footnotesize {\small\bf Abstract:} {\small
In this paper, we study the Dirichlet problem of Hessian quotient equations in exterior domains. By estimating the eigenvalues of the solution, the necessary and sufficient conditions on existence of radial solutions are obtained. Applying the solutions of ODE, the viscosity subsolutions and supersolutions are constructed and then the existence of viscosity solutions for exterior problems is established by the Perron's method.}}

 {\small{\bf Keywords:} Hessian quotient equations; exterior Dirichlet problem; radial solutions; asymptotic behavior; necessary and sufficient conditions}

{\small{\bf 2020 MSC.} 35J60, 35B40}

\end{minipage}

\section{Introduction}


In this paper, we study the exterior Dirichlet problem of the Hessian quotient equations
{\s\begin{eqnarray}
\vspace{2mm} \label{hq1}\frac{S_k(D^2u)}{S_l(D^2u)}&=&g(x)\
\ \mbox{in}\ \
\mathbb{R}^n\backslash\o{\Omega},\\
\vspace{2mm} \label{hq2}u&=&\phi\ \ \mbox{on}\ \ \partial \Omega,
\end{eqnarray}}where $\Omega$ is a bounded set in $\mathbb{R}^n$, $0\leq l<k\leq
n,n\geq 2$, $S_j(D^2u)$
is the $j$th elementary symmetric function of the
eigenvalues $\lambda=(\lambda_1,\lambda_2,\cdots,\lambda_n)$ of Hessian matrix
$D^2u$, i.e.,
$$S_j(D^2u)=\sigma_j(\lambda(D^2u))=\sum_{1\leq
i_1<\cdots<i_j\leq
n}\lambda_{i_1}\cdots\lambda_{i_j},j=1,2,\cdots,n,$$
and $g\in C^{0}(\mathbb{R}^{n}\backslash\Omega)$ is positive, $ \phi\in C^{2}(\partial\Omega)$.
For $l=0$, we set $S_0(D^2u)\equiv 1.$

 The Hessian quotient equation \eqref{hq1} is fully
 nonlinear and is closely connected with the geometric
 problem. Some well-known equations are its special cases. If $l=0$, \eqref{hq1} is the $k$-Hessian equation. Particularly, \eqref{hq1} is the Poisson
 equation if $k=1,l=0$, while it is the Monge-Amp\`{e}re equation for $k=n,l=0$.
 If $k=n=3,l=1$, \eqref{hq1} becomes the special Lagrangian equation $\mbox{det} D^2u=\Delta u$ which
stems from special Lagrangian geometry \cite{HL}: let $u$ be a solution of \eqref{hq1}, then the graph of
 $Du$ over $\mathbb{R}^3$ in $\mathbb{C}^3$ is a special Lagrangian submanifold in $\mathbb{C}^3$, that is, its mean curvature
 vanishes everywhere and the complex structure on $\mathbb{C}^3$ sends the tangent
 space of the graph to the normal space at every point. So \eqref{hq1} has attracted a lof of attention, see \cite{BCGJ, CNS, LB, T} and the references therein.

 A classical theorem of Monge-Amp\`{e}re equation states that any convex classical solution of $\det D^2u=1$ in $\mathbb{R}^n$ must be a quadratic polynomial which is proved by
 J\"{o}rgens \cite{j} ($n=2$), Calabi \cite{c} ($n\leq 5$) and Pogorelov \cite{p} ($n\geq 2$). One can also refer to the related results \cite{cy, ca, jx, tw1}. In 2003, Caffarelli and Li \cite{CL} made the extension of J\"{o}rgens-Calabi-Pogorelov theorem and also investigated the existence of solutions for the exterior Dirichlet problem of Monge-Amp\`{e}re equation
\begin{equation}\label{ma}\begin{cases}\det D^2u=1,\quad x\in\mathbb{R}^n\b\o{\Omega},\\
u=\phi,\quad x\in\partial\Omega.
\end{cases}\end{equation}
They obtained that if $\Omega$ is a smooth, bounded, strictly convex open subset and $\phi\in C^2(\partial\Omega)$, then for any given $b\in \mathbb{R}^n$ and any given $n\times n$ real symmetric positive definite matrix $A$ with $\det A=1$, there exists some constant $c^*$ depending only on $n,\Omega,\phi,b$ and $A$, such that for every $c>c^*$ there exists a unique function $u\in C^{\infty}(\mathbb{R}^n\b\o{\Omega})\cap C^0(\o{\mathbb{R}^n\b\Omega})$ which satisfies \eqref{ma} and the asymptotic behavior at infinity
$$\limsup_{|x|\to\infty}\left(|x|^{n-2}\left|u(x)-\left(\frac{1}{2}x^TAx+b\cdot x+c\right)\right|\right)<\infty.$$
Since then, there have been extensive studies of the exterior problem for the fully nonlinear elliptic equations. In 2011, Dai and Bao \cite{DB} ($n\geq 3$), Dai \cite{d1} ($n\geq 3$) studied the exterior Dirichlet problem of Hessian equation
\begin{equation}\label{Hessian}S_k(D^2u)=1
\end{equation}
 and got the existence and uniqueness of the viscosity solutions with the asymptotic behavior.
 Bao, Li and Li \cite{BLL} ($n\geq 3$), Cao-Bao \cite{CB} ($n\geq 3$) studied the viscosity solutions with the generalized asymptotic behavior of the exterior Dirichlet problem for Hessian equation. The subsequent results following \cite{CL} of the exterior Dirichlet problem for Monge-Amp\`{e}re equations can be referred to \cite{BL01,BLZ1,BLZ2,bxz,h,LL}; For the Hessian quotient equations
$$\frac{S_{k}(D^{2}u)}{S_{l}(D^{2}u)}=1,$$
where $0\leq l<k\leq n, n\geq 3$, the first author \cite{D} obtained the existence of viscosity solutions with the asymptotic behavior
\begin{equation}\label{hq-asy}\limsup_{|x|\to\infty}\left(|x|^{k-l-2}\left|u(x)-\left(\frac{\hat{a}}{2}|x|^2+c\right)\right|\right)<+\infty\end{equation}
for the exterior Dirichlet problem with $\hat{a}=(C_n^l/C_n^k)^{\frac{1}{k-l}}, k-l\geq 3$. Li and the first author \cite{LD} studied the viscosity solutions with the asymptotic behavior \eqref{hq-asy} for $k\leq (n+1)/2$ which includes $k-l=1$ and $k-l=2$ not included in \cite{D}. Recently, Li and Li \cite{LLX1} considered the viscosity solutions with the generalized asymptotic behavior
$$\limsup_{|x|\to\infty}\left(|x|^{m-2}\left|u(x)-\left(\frac{1}{2}x^{T}Ax+b^Tx+c\right)\right|\right)<+\infty$$
with $m\in (2,n]$ and $S_k(A)/S_l(A)=1.$ One can also refer to \cite{llz}. We would like to remark that for $n=2$, the exterior Dirichlet problem of Monge-Amp\`{e}re equations  can be referred to the earlier works by Ferrer, Mart\'{i}nez, and Mil\'{a}n \cite{FMM1,FMM2} using the complex variable methods and the works by Delano\"{e} \cite{De}.

For convenience, we first restrict the class of functions. Define
 $$\Gamma_k=\{\lambda\in {\bf{\mathbb{R}}}^n|\sigma_j(\lambda)>0,j=1,2,\cdots,k\}.$$
Let $u\in C^2(\mathbb{R}^n\backslash \o{\Omega})$ and $\lambda=\lambda(D^2u)=(\lambda_1,\lambda_2,\cdots,\lambda_n)$ be
the eigenvalues of the Hessian matrix $D^2u$. If $\lambda\in \o{\Gamma_k}(\Gamma_k)$, then we call $u$ is $k$-convex (uniformly $k$-convex). If $k=n$, a uniformly $n$-convex function is a convex function. Eq. \eqref{hq1} is elliptic for uniformly $k-$convex functions.

For $m=1,2,\dots,n$, let
$$\Phi_m=\{u\in C^1(\mathbb{R}^n\b B_1)\cap C^2(\mathbb{R}^n\b \o{B_1})|u\ \mbox{is uniformly}\ m\mbox{-convex and radially symmetric.}\}$$

 Wang and Bao \cite{WB} studied the necessary and sufficient conditions on existence and convexity of radial solutions to the exterior Dirichlet problems of Hessian equations
 \begin{equation}\label{hh1}
 \left\{
 \begin{array}{rll}
 S_k(D^2u)&=&1,\ x\in \mathbb{R}^n\b\o{B_1},\\
u&=&\bar{b},\ x\in \partial B_1,\\
u&=&\dfrac{\bar{a}}{2}|x|^2+\bar{c}+O(|x|^{2-n}),\ |x|\to\infty
 \end{array}
 \right.
 \end{equation}
 with $\bar{b},\bar{c}\in \mathbb{R}, \bar{a}=(1/C_n^k)^{\frac{1}{k}}$. Lately, Li and Lu \cite{LL} characterized the existence and nonexistence of solutions for exterior problem of Monge-Amp\`{e}re equations.

  As far as we know for Hessian quotient equations, there is no such beautiful result. In this paper, we first consider the necessary and sufficient conditions on existence of radial solutions to the exterior Dirichlet problems of Hessian quotient equations with prescribed asymptotic behavior at infinity.

  Let $B_1$ be the unit ball in $\mathbb{R}^n$. The main results of this paper are the following.

\begin{theorem}\label{thm-ns}
Let $2\leq k\leq m\leq n, n\geq 3,0\leq l<k$, and $b\in \mathbb{R}$. The exterior Dirichlet problem
\begin{align}
\label{1eq1}\dfrac{S_{k}(D^2u)}{S_{l}(D^2u)}=&1\
\ \mbox{in}\ \
\mathbb{R}^n\backslash\o{B_{1}},\\
\label{1eq2}u=&b\ \ \mbox{on}\ \ \partial B_1
\end{align}
has a unique function $u\in \Phi_m$ satisfying
$$u(x)=\dfrac{\hat{a}}{2}|x|^2+c+O(|x|^{2-n}),\ |x|\to\infty$$
if and only if $c\in[\mu(\alpha_1),+\infty)$ for $m=k$, and $c\in[\mu(\alpha_1),\mu(\alpha_2)]$ for $m>k$, where $\hat{a}=(C_n^l/C_n^k)^{\frac{1}{k-l}},$
$$c=\mu(\alpha)=b-\dfrac{\hat{a}}{2}+\hat{a}\dint_1^{\infty}s
\left[\left(1+\dfrac{\alpha}{C_n^ls^nU^l(s,\alpha)}\right)^{\frac{1}{k-l}}-1\right]ds,$$
\begin{equation}\label{alpha1}\alpha_1:=\sup_{r>1}U^{-1}_{r}\left(\left(\dfrac{lC_n^l}{kC_n^k}\right)^{\frac{1}{k-l}}\right)\end{equation}
\begin{equation}\label{alpha2}\alpha_2:=\inf_{r>1}U^{-1}_{r}\left(\left(\dfrac{(m-l)C_n^l}{(m-k)C_n^k}\right)^{\frac{1}{k-l}}\right),\end{equation}
and $U:=U(r,\alpha):=U_r(\alpha)$ satisfies
$$U^k-\frac{C_n^l}{C_n^k}U^l-\frac{\alpha}{C_n^kr^n}=0,\ r>1,$$
and \begin{equation}\label{Usup}U>\left(\dfrac{lC_n^l}{kC_n^k}\right)^{\frac{1}{k-l}}.\end{equation}
\end{theorem}

\begin{remark}\label{nsrem1}
If $l=0$, for the case of Hessian equation $S_k(D^2u)=1$, we then get that $\alpha_1=-1,\alpha_2=\frac{k}{m-k}.$ So Theorem \ref{thm-ns} corresponds to Theorem 2 in \cite{WB}.
\end{remark}

\begin{remark}
For the Hessian equations, the eigenvalue $\gamma$ of Hessian matrix $D^2u$ of the solutions $u$ satisfies $\gamma^k=\frac{1+\alpha r^{-n}}{C_n^k}$ with $\alpha=C_n^k(u^{\prime}(1))^k-1$, so $\alpha$ is relatively easy to be estimated. Unlike the Hessian equations, in this paper, the eigenvalue $\gamma$ satisfies $\gamma^k-\frac{C_n^l}{C_n^k}\gamma^l-\frac{\alpha}{C_n^kr^n}=0$ with $\alpha=C_n^k(u^{\prime}(1))^k-C_n^l(u^{\prime}(1))^l$, it is difficult to estimate $\alpha$, here we use the inverse function, see Lemma \ref{lemma2}.
\end{remark}

Let $k=n=2, l=1$, then the exterior Dirichlet problem \eqref{1eq1}, \eqref{1eq2} becomes
\begin{eqnarray}
\label{sl2eq1}\det D^2u&=&\Delta u\
\ \mbox{in}\ \
\mathbb{R}^2\backslash\o{B_{1}},\\
\label{sl2eq2}u&=&b\ \ \mbox{on}\ \ \partial B_1.
\end{eqnarray}
\begin{theorem}\label{hqn2}
The exterior Dirichlet problem \eqref{sl2eq1}, \eqref{sl2eq2} has a unique solution $u\in \Phi_2$ satisfying
\begin{equation}\label{u22}u(x)=|x|^2+\dfrac{\rho}{2}\ln |x|+c+O(|x|^{-2})\end{equation}
if and only if $\rho\geq -1$, where $c=\nu(\rho)$ and
$$\nu(\rho)=b-\dfrac{1}{2}+\dfrac{\rho}{4}+\dfrac{\rho}{2}\ln 2-\dfrac{1}{2}[\sqrt{1+\rho}+\rho \ln(1+\sqrt{1+\rho})].$$
\end{theorem}

Next we study the existence of solutions of problem \eqref{hq1} and \eqref{hq2}. Now we first recall the definition of viscosity solution.
\begin{definition}
\par A function $u\in C^0(\mathbb{R}^{n}\setminus\o{\Omega})$ is said to be a viscosity subsolution of \eqref{hq1}, if for any $\bar{x}\in\mathbb{R}^{n}\setminus\o{\Omega}$, $ v\in C^{2}(\mathbb{R}^{n}\setminus\o{\Omega})$ satisfying
$$u(x)\leq v(x),~x\in\mathbb{R}^{n}\setminus\o{\Omega}\quad\mbox{and}\quad u(\bar{x})= v(\bar{x}),$$we have$$S_{k}(D^{2} v(\bar{x}))\geq g(\bar{x}).$$
\par Similarly, $u\in C^0(\mathbb{R}^{n}\setminus\o{\Omega})$ is said to be a viscosity supersolution of \eqref{hq1}, if for every $\bar{x}\in\mathbb{R}^{n}\setminus\o{\Omega}$, $k-$convex function $ v\in C^{2}(\mathbb{R}^{n}\setminus\o{\Omega})$ satisfying
$$u(x)\geq v(x),~x\in\mathbb{R}^{n}\setminus\o{\Omega}\quad\mbox{and}\quad u(\bar{x})= v(\bar{x}),$$we have $$S_{k}(D^{2} v(\bar{x}))\leq g(\bar{x}).$$
\par If $u\in C^{0}(\mathbb{R}^{n}\setminus\o{\Omega})$ is both a viscosity subsolution and a viscosity supersolution of \eqref{hq1}, we call that $u$ is a viscosity solution of \eqref{hq1}.

 If $u$ is a viscosity subsolution (resp. supersolution, solution) to \eqref{hq1} and $u\leq$(resp. $\geq,=)\phi(x)$ on $\partial\Omega$, we call that $u$ is a viscosity subsolution (resp. supersolution, solution) of \eqref{hq1} and \eqref{hq2}.
\end{definition}

Let $S(n)$ denote the set of the real $n\times n$ symmetric matrices, and for $A\in S(n)$, $A$ is positive definite, let $r:=r_A(x):=\sqrt{x^{T}Ax}$. Suppose that $g_0\in C^{0}([0,+\infty))$ is a positive function of $r$,
$$0<\inf_{[0,+\infty)}g_0\leq \sup_{[0,+\infty)}g_0<+\infty,$$
and for some constant $\u{C}_0>0$,
$$\lim_{r\to+\infty}g_0(r)=\u{C}_0>0.$$
Assume that $g\in C^0(\mathbb{R}^n)$ satisfies that
\begin{equation}\label{gbdd}0<\inf_{\mathbb{R}^n}g\leq \sup_{\mathbb{R}^n}g<+\infty,\end{equation}
and there exists a constant $\beta>2$ such that
\begin{equation}\label{4}
g(x)=g_0(r)+O(r^{-\beta}),\ |x|\to\infty.
\end{equation}
Define
$$\Lambda_k(A,x)=\dfrac{\sum_{i=1}^n\frac{\partial}{\partial\lambda_i}
\sigma_{k}(\lambda(A))\lambda_i^2x_i^2}{\sigma_k(\lambda(A))\sum_{i=1}^n\lambda_ix_i^2},$$ and
	\begin{equation*}
\o{t}_{k}:=\o{t}_{k}(A):=\sup\limits_{x\in\mathbb{R}^n\b\{0\}}\Lambda_k(A,x),
	\end{equation*}
$$\u{t}_{k}:=\u{t}_{k}(A):=\inf\limits_{x\in\mathbb{R}^n\b\{0\}}\Lambda_k(A,x),$$
where $\lambda_i,i=1,\dots,n$ are the eigenvalues of the matrix $A$.

\begin{remark}
If $A=\bar{C}I$ for some $\bar{C}>0$, then $\o{t}_k=\u{t}_k=\Lambda_k=k/n$. See also Lemma \ref{lem5}.
\end{remark}
Let
$$\mathcal{A}_{k,l}=\left\{A|A\in S(n), A  \mbox{\ is positive definite and satisfy}\ \frac{S_k(A)}{S_l(A)}=1\right\}.$$

\begin{theorem}\label{thm1} Let $0\leq l<k\leq n, n\geq3,$ $\frac{k-l}{\o{t}_k-\u{t}_l}>2$ and $\Omega$ be a bounded, strictly convex subset in $\mathbb{R}^{n},\phi\in C^{2}(\partial\Omega)$. Assume that $g$ satisfies \eqref{gbdd} and \eqref{4}.
Then for any given $b\in\mathbb{R}^{n}$ and $A\in \mathcal{A}_{k,l}$, there exists some constant $\tilde{c}$, depending only on $n,~b,~A,~\Omega,~g,~g_0,$ $||\phi||_{C^{2}(\partial\Omega)}$, such that for every $c>\tilde{c}$, there exists a unique viscosity solution $u\in C^{0}(\mathbb{R}^{n}\backslash\Omega)$ of \eqref{hq1} and \eqref{hq2} satisfying
\begin{align}\label{asymptotic1}
&\limsup\limits_{|x|\rightarrow\infty}\left(
|x|^{\min\{\beta,\frac{k-l}{\o{t}_k-\u{t}_l}\}-2}\left|u(x)-\left(u_0(x)+b\cdot x+c\right)\right|\right)<\infty,\ \mbox{if}\ \beta\not=\frac{k-l}{\o{t}_k-\u{t}_l},
\end{align}
or
\begin{align}\label{asymptotic2}
&\limsup\limits_{|x|\rightarrow\infty}\left(
|x|^{\min\{\beta,\frac{k-l}{\o{t}_k-\u{t}_l}\}-2}(\ln|x|)^{-1}\left|u(x)-\left(u_0(x)+b\cdot x+c\right)\right|\right)<\infty,\ \mbox{if}\ \beta=\frac{k-l}{\o{t}_k-\u{t}_l},
\end{align}
where
$$u_0(x)=\int_{0}^{r_A(x)}\theta h_0(\theta)d\theta,$$
and $h_0$ satisfies
\begin{equation}\label{Q5}
h_0(r)=\left(r^{-\frac{k-l}{\o{t}_k-\u{t}_l}}
\dint_0^r g_0(s)(s^{\frac{k-l}{\o{t}_k-\u{t}_l}}h_0(s)^{\frac{(k-l)\u{t}_l}{\o{t}_k-\u{t}_l}})^{\prime}ds\right)^{\frac{\o{t}_k-\u{t}_l}{(k-l)\o{t}_k}}.\end{equation}
\end{theorem}

The main difficulty of proving Theorem \ref{thm1} is not only to construct the generalized symmetric subsolutions with asymptotic behavior, but also to construct the generalized symmetric supersolutions with asymptotic behavior which is different from the case $g\equiv 1$ in \cite{LLX1} where $\frac{1}{2}x^TAx+b\cdot x+c$ are directly the supersolutions. In addition, we use different method from which in \cite{LLX1} to obtain the asymptotic behavior of subsolutions and supersolutions on the basis of the expressions of subsolutions and supersolutions themselves.

\begin{remark}\label{rem2}
If $g_0\equiv1$, then $g=1+O(r^{-\beta}),r\to+\infty$ and by \eqref{Q5}, $h_0\equiv 1$, and so
\begin{align*}\label{asymptotic3}
u_0(x)=\frac{1}{2}x^{T}Ax.
\end{align*}
\end{remark}

\begin{remark}
For the special case that $l=0,k\geq 2$, the Hessian equation $S_k(D^2u)=g(x)$ with $g(x)=g_0(r)+O(r^{-\beta})$, our Theorem \ref{thm1} is Theorem 1.3 in \cite{d2}. For the special case that $l=0, g_0\equiv 1$, the Hessian equation $S_k(D^2u)=g(x)$ with $g(x)=1+O(|x|^{-\beta})$, our Theorem \ref{thm1} extends the corresponding Theorem 1.1 in \cite{CB}, where the solution $u$ satisfies the asymptotic behavior
\begin{align*}
&\limsup\limits_{|x|\rightarrow\infty}\left(
|x|^{\min\{\beta,\frac{k}{H_k}\}-2}\left|u(x)-\left(\frac{1}{2}x^{T}Ax+b\cdot x+c\right)\right|\right)<\infty,\ \mbox{if}\ \beta\not=\frac{k}{H_k},
\end{align*}
or
\begin{align*}
&\limsup\limits_{|x|\rightarrow\infty}\left(
|x|^{\min\{\beta,\frac{k}{H_k}\}-2}(\ln|x|)^{-1}\left|u(x)-\left(\frac{1}{2}x^{T}Ax+b\cdot x+c\right)\right|\right)<\infty,\ \mbox{if}\ \beta=\frac{k}{H_k},
\end{align*}
with $k\geq 2, H_k=H_k(A):=\displaystyle\max_{1\leq i\leq n}A_k^i(A)$, $A_k^i(A)=\lambda_i(A)\frac{\partial}{\partial \lambda_i}\sigma_k(\lambda(A)),i=1,\dots,n$. But in our paper, if $l=0$, then $\u{t}_l=0,$ and our $\o{t}_k$ is different from $H_k$. For $A\in \mathcal {A}_{k,0}$, we can know that $\o{t}_k\leq H_k$ which means that the result of Theorem \ref{thm1} is better than that of Theorem 1.1 in \cite{CB}.

For the special case that $l=0,k=n,g_0\equiv 1$, the Monge-Amp\`{e}re equation $\det D^2u=g(x)$ with $g(x)=1+O(|x|^{-\beta})$, our result \eqref{asymptotic1} is included in the Corollary 1.1 in \cite{BLZ1} but where \eqref{asymptotic2} is missed.
\end{remark}

\begin{remark}
In Theorem \ref{thm1}, since $\frac{k-l}{\o{t}_k-\u{t}_l}>2$, $k-l\geq 1$ and $0<\o{t}_k-\u{t}_l<1,$ then the cases that $k-l=1,\o{t}_k-\u{t}_l\geq\frac{1}{2}$ are not included. For the two cases, the supersolutions with asymptotic behavior are not easy to seek. Maybe we need to look for new methods.
\end{remark}


The remainders of this paper are arranged as follows. In section 2, we will prove Theorems \ref{thm-ns} and \ref{hqn2}. With the help of the inverse function, we estimate $\alpha=C_n^k(u^{\prime}(1))^k-C_n^l(u^{\prime}(1))^l$, see Lemma \ref{lemma2}.
In Section 3, we first prove Proposition \ref{prop1} which indicates that there does not exist generalized symmetric solution (see Definition \ref{Definition}) for the Hessian quotient equation \eqref{hq1} with $0\leq l<k\leq n$ unless $A=\hat{a}I$. Then we construct a family of generalized symmetric smooth $k-$convex subsolutions and supersolutions of \eqref{hq1} by the solutions of ODE and in the fourth section, we will prove Theorem \ref{thm1} by Perron's method. Finally, in Appendix, we give several lemmas which are used in this paper.

\section{Proof of Theorem \ref{thm-ns}}

\begin{lemma}\label{lem-w}
Let $\lambda=(\omega,\gamma,\dots,\gamma)\in \Gamma_m, n\geq m\geq 2.$ Then $\gamma>0.$
\end{lemma}

\begin{proof}
See Lemma 1  in \cite{WB}.
\end{proof}

\begin{lemma}\label{lemma1}
Let $\lambda=(\omega,\gamma,\dots,\gamma)\in \mathbb{R}^{n}$ and $\sigma_{k,l}(\lambda)=1,$ $2\leq k\leq n, 0\leq l<k.$ Then $\lambda\in \Gamma_m (k\leq m\leq n)$ if and only if
\begin{equation}\label{1}
\left(\dfrac{lC_n^l}{kC_n^k}\right)^{\frac{1}{k-l}}<\gamma<\gamma_m,
\end{equation}
where
\begin{equation}\label{gammam}\gamma_m=\left\{
\begin{array}{rll}
\left(\dfrac{(m-l)C_n^l}{(m-k)C_n^k}\right)^{\frac{1}{k-l}},&m>k\\
+\infty,&m=k.
\end{array}
\right.
\end{equation}
\end{lemma}

\begin{proof}
Since $\sigma_{k,l}(\lambda)=1$, then
$$\dfrac{C_{n-1}^{k-1}\omega\gamma^{k-1}+C_{n-1}^k\gamma^k}{C_{n-1}^{l-1}\omega\gamma^{l-1}+C_{n-1}^l\gamma^l}=1.$$
So
\begin{equation}\label{eq3}\omega=\dfrac{(n-l)C_n^l\gamma^l-(n-k)C_n^k\gamma^k}{kC_n^k\gamma^{k-1}-lC_n^l\gamma^{l-1}}.\end{equation}
Due to $\lambda\in \Gamma_m$, by the definition of $\Gamma_m$, we get that $$\sigma_j(\lambda)=C_{n-1}^{j-1}\omega\gamma^{j-1}+C_{n-1}^j\gamma^j>0,j=1,2,\dots,m,$$
and then
$$\gamma^{j-1}(j\omega+(n-j)\gamma)>0.$$
By Lemma \ref{lem-w}, since $\gamma>0$, we have
$$j\omega+(n-j)\gamma>0.$$
From \eqref{eq3}, we know that
$$j\omega+(n-j)\gamma=\dfrac{n(C_n^l(j-l)\gamma^l-C_n^k(j-k)\gamma^k)}{kC_n^k\gamma^{k-1}-lC_n^l\gamma^{l-1}}>0.$$
Then we have two cases.

Case 1: $$\left\{\begin{array}{rll}
C_n^l(j-l)\gamma^l&>&C_n^k(j-k)\gamma^k,\\
kC_n^k\gamma^{k-1}&>&lC_n^l\gamma^{l-1}.
\end{array}\right.$$

Case 2: $$\left\{\begin{array}{rll}
C_n^l(j-l)\gamma^l&<&C_n^k(j-k)\gamma^k,\\
kC_n^k\gamma^{k-1}&<&lC_n^l\gamma^{l-1}.
\end{array}\right.$$

For the Case 1, if $m=k$, then
$$\left\{\begin{array}{rlll}
\gamma^{k-l}&>&\dfrac{lC_n^l}{kC_n^k},&j<k,\\
\gamma^{k-l}&>&\dfrac{lC_n^l}{kC_n^k},&j=k.
\end{array}\right.$$
 So if $m=k$, we have
\begin{equation}\label{eq4}\gamma^{k-l}>\frac{lC_n^l}{kC_n^k}.\end{equation} If $m>k$, then
$$\left\{\begin{array}{rlll}
\gamma^{k-l}&>&\dfrac{lC_n^l}{kC_n^k},&j\leq k,\\
\dfrac{lC_n^l}{kC_n^k}&<&\gamma^{k-l}< \dfrac{(m-l)C_n^l}{(m-k)C_n^k},&j\geq k+1.
\end{array}\right.$$
So if $m>k$, we have
\begin{equation}\label{eq5}\frac{lC_n^l}{kC_n^k}<\gamma^{k-l}< \dfrac{(m-l)C_n^l}{(m-k)C_n^k}.\end{equation}

For the Case 2, whether $m=k$ or $m>k$, Case 2 is impossible. The Lemma is proved.
\end{proof}

\begin{lemma}\label{lemma2}
Let $2\leq k\leq m\leq n, 0\leq l<k,$ and suppose that $u\in C^1(\mathbb{R}^{n}\b B_1)\cap C^2(\mathbb{R}^n\b\o{B_1})$ is a radial solution of \eqref{1eq1} and \eqref{1eq2}. Set
$$\alpha=C_n^k(u^{\prime}(1))^k-C_n^l(u^{\prime}(1))^l.$$ Then $u$ is uniformly $k-$convex if and only if $\alpha\in [\alpha_1,+\infty)$, and $u$ is uniformly $m(m>k)-$convex if and only if $\alpha\in[\alpha_1,\alpha_2]$, where $\alpha_1,\alpha_2$ are given by \eqref{alpha1}, \eqref{alpha2}.

\end{lemma}

\begin{proof} Let $u(x)=u(r)=u(|x|)\in C^1(\mathbb{R}^n\b B_1)\cap C^2(\mathbb{R}^n\b \o{B_1})$ be a radial solution
of \eqref{1eq1} and \eqref{1eq2}.
Then
$$D_{ij}u=(ru^{\prime\prime}-u^{\prime})\dfrac{x_ix_j}{r^3}+u^{\prime}\dfrac{\delta_{ij}}{r}, i,j=1,\dots,n, r>1.$$
So the eigenvalues of the Hessian matrix $D^2u$ are
\begin{equation}\label{2eq1}\lambda_1=u^{\prime\prime},\lambda_2=\cdots=\lambda_n=\dfrac{u^{\prime}}{r}.\end{equation}
In view of Lemma \ref{lem-w}, we know that
$$\gamma=\dfrac{u^{\prime}}{r}\geq 0, r\geq 1.$$

Moreover,
$$\sigma_k(\lambda(D^2u))=u^{\prime\prime}C_{n-1}^{k-1}\left(\dfrac{u^{\prime}}{r}\right)^{k-1}+C_{n-1}^k\left(\dfrac{u^{\prime}}{r}\right)^{k}.$$
Hence
$$\dfrac{u^{\prime\prime}C_{n-1}^{k-1}\left(\frac{u^{\prime}}{r}\right)^{k-1}+C_{n-1}^k\left(\frac{u^{\prime}}{r}\right)^{k}}
{u^{\prime\prime}C_{n-1}^{l-1}\left(\frac{u^{\prime}}{r}\right)^{l-1}+C_{n-1}^l\left(\frac{u^{\prime}}{r}\right)^{l}}=1,$$
i.e.,
\begin{equation}\label{radial}\dfrac{C_n^k(r^{n-k}(u^{\prime})^k)^{\prime}}{C_n^l(r^{n-l}(u^{\prime})^l)^{\prime}}=1.\end{equation}
Then integrating the above equality from $1$ to $r$, we have
\begin{equation}\label{eq6}(u^{\prime})^{k-l}=\dfrac{C_n^l}{C_n^k}r^{k-l}\left(1+\dfrac{\alpha}{C_n^lr^n(\frac{u^{\prime}}{r})^l}\right).\end{equation}
That is
$$\left(\frac{u^{\prime}}{r}\right)^k-\frac{C_n^l}{C_n^k}\left(\frac{u^{\prime}}{r}\right)^l-\frac{\alpha}{C_n^kr^n}=0.$$
Let $U(r):=u^{\prime}(r)/r$. Then $U$ satisfies
\begin{equation}\label{Uequation}U^k-\frac{C_n^l}{C_n^k}U^l-\frac{\alpha}{C_n^kr^n}=0,\ r>1.\end{equation}
Set $$F(U,\alpha)=U^k-\frac{C_n^l}{C_n^k}U^l-\frac{\alpha}{C_n^kr^n}.$$
Then since by \eqref{1}, $U^{k-l}>lC_n^l/kC_n^k,$ thus
$$\frac{\partial F}{\partial U}=kU^{k-1}-l\frac{C_n^l}{C_n^k}U^{l-1}>0.$$
Due to the implicit function theorem, \eqref{Uequation} has a unique function $U=U(r,\alpha)$ and
\begin{equation}\label{U}\frac{\partial U}{\partial\alpha}=\frac{1}{r^n(kC_n^kU^{k-1}-lC_n^lU^{l-1})}>0.\end{equation}
And in virtue of \eqref{Uequation}, for $r>1$,
\begin{equation}\label{Uinfty}\lim_{\alpha\to+\infty}U(r,\alpha)=+\infty.\end{equation}
Set $U_{r}(\alpha):=U(r,\alpha)$. Then by \eqref{U}, the inverse function of $U_{r}(\alpha)$ exists, we denote it as $U^{-1}_{r}(\alpha)$. Moreover, $dU^{-1}_r/d\alpha>0.$
So by Lemma \ref{lemma1}, $u$ is uniformly $m-$convex if and only if
$$\left(\dfrac{lC_n^l}{kC_n^k}\right)^{\frac{1}{k-l}}<U_{r}(\alpha)<\gamma_m.$$
As a result, $$U^{-1}_{r}\left(\left(\dfrac{lC_n^l}{kC_n^k}\right)^{\frac{1}{k-l}}\right)<\alpha<U^{-1}_{r}(\gamma_m),\ r>1,$$
which is equivalent to
$$\alpha_1:=\sup_{r>1}U^{-1}_{r}\left(\left(\dfrac{lC_n^l}{kC_n^k}\right)^{\frac{1}{k-l}}\right)\leq \alpha<+\infty,\ \mbox{if}\ m=k,$$
and
$$\alpha_1:=\sup_{r>1}U^{-1}_{r}\left(\left(\dfrac{lC_n^l}{kC_n^k}\right)^{\frac{1}{k-l}}\right)\leq \alpha\leq \inf_{r>1}U^{-1}_{r}\left(\left(\dfrac{(m-l)C_n^l}{(m-k)C_n^k}\right)^{\frac{1}{k-l}}\right):=\alpha_2,\ \mbox{if}\ m>k.$$
Lemma \ref{lemma2} is proved.
\end{proof}

{\bf Proof of Theorem \ref{thm-ns}}\ By \eqref{eq6},
\begin{equation}\label{Ueq}U^{k-l}(r,\alpha)=\dfrac{C_n^l}{C_n^k}\left(1+\dfrac{\alpha}{C_n^lr^nU^l(r,\alpha)}\right).\end{equation}
So $$\left(\frac{u^{\prime}}{r}\right)^{k-l}=\dfrac{C_n^l}{C_n^k}\left(1+\dfrac{\alpha}{C_n^lr^nU^l(r,\alpha)}\right).$$
Integrating the above equality on both sides for $r$, we have
\begin{equation}\label{u1}\begin{array}{rlll}
u(x)&=&u(1)+\left(\dfrac{C_n^l}{C_n^k}\right)^{\frac{1}{k-l}}\dint_1^{|x|}
\left\{s\left[\left(1+\dfrac{\alpha}{C_n^ls^nU^l(s,\alpha)}\right)^{\frac{1}{k-l}}-1\right]+s\right\}ds\\
&=&\dfrac{\hat{a}}{2}|x|^2+b-\dfrac{\hat{a}}{2}+\hat{a}\dint_1^{\infty}s\left[\left(1+\dfrac{\alpha}{C_n^ls^nU^l(s,\alpha)}\right)^{\frac{1}{k-l}}-1\right]ds\\
&&-\hat{a}\dint_{|x|}^{\infty}s\left[\left(1+\dfrac{\alpha}{C_n^ls^nU^l(s,\alpha)}\right)^{\frac{1}{k-l}}-1\right]ds\\
&=&\dfrac{\hat{a}}{2}|x|^2+c-\hat{a}\dint_{|x|}^{\infty}s\left[\left(1+\dfrac{\alpha}{C_n^ls^nU^l(s,\alpha)}\right)^{\frac{1}{k-l}}-1\right]ds,
\end{array}\end{equation}
where $\hat{a}=(C_n^l/C_n^k)^{\frac{1}{k-l}},$ and
\begin{equation}\label{c1}c=\mu(\alpha)=b-\dfrac{\hat{a}}{2}+\hat{a}\dint_1^{\infty}s
\left[\left(1+\dfrac{\alpha}{C_n^ls^nU^l(s,\alpha)}\right)^{\frac{1}{k-l}}-1\right]ds.\end{equation}
Since by Lemma \ref{lemma1}, we know that $U$ has a lower bound. So, for $n\geq 3$, the infinite integral in \eqref{c1} is convergent. Let
$$V(r,\alpha)=\frac{\alpha}{U^l(r,\alpha)}.$$
Then by \eqref{U},
\begin{align*}\frac{\partial V}{\partial \alpha}=&U^{-l-1}\left(U-\alpha l\frac{\partial U}{\partial \alpha}\right)\\
=&U^{-l-1}\dfrac{r^n(kC_n^kU^{k}-lC_n^lU^{l})-l\alpha}{r^n(kC_n^kU^{k}-lC_n^lU^{l})}.
\end{align*}
From \eqref{Ueq}, $\alpha=r^n(C_n^kU^{k}-C_n^lU^{l})$. By Lemma \ref{lemma1}, $U^{k-l}>lC_n^l/kC_n^k$, so $\partial V/\partial \alpha>0.$ Then $\mu$ is strictly increasing in $\alpha$. Based on \eqref{Uinfty} and \eqref{Uequation}, we know that for $r>1$,
$$\lim_{\alpha\to+\infty}\dfrac{\alpha}{C_n^lr^nU^l(r,\alpha)}=+\infty,$$
and then $\mu(+\infty)=+\infty.$ Thus by Lemma \ref{lemma2}, we have that $u$ is uniformly $m-$convex if and only if $c\in [\mu(\alpha_1),+\infty)$ for $m=k$, and $c\in[\mu(\alpha_1),\mu(\alpha_2)]$ for $m>k$.

Moreover, by Lemma \ref{lemma1}, since $U^{k-l}>lC_n^l/kC_n^k$, then
$$\hat{a}\dint_{|x|}^{\infty}s\left[\left(1+\dfrac{\alpha}{C_n^ls^nU^l(s,\alpha)}\right)^{\frac{1}{k-l}}-1\right]ds=O(|x|^{2-n}), |x|\to\infty.$$
As a result, by \eqref{u1},
$$u(x)=\dfrac{\hat{a}}{2}|x|^2+c+O(|x|^{2-n}),\ |x|\to\infty.$$
Then Theorem \ref{thm-ns} is proved. $\hfill\Box$

If $k=n=m=3, l=1$, then the exterior Dirichlet problem \eqref{1eq1}, \eqref{1eq2} becomes the problem of special Lagrangian equation
\begin{eqnarray}
\label{sleq1}\det D^2u&=&\Delta u\
\ \mbox{in}\ \
\mathbb{R}^3\backslash\o{B_{1}},\\
\label{sleq2}u&=&b\ \ \mbox{on}\ \ \partial B_1.
\end{eqnarray}
According to \eqref{radial}, the radially symmetric solution of the problem \eqref{sleq1} and \eqref{sleq2} satisfies
\begin{equation}\label{srad}
(u^{\prime})^3-3r^2u^{\prime}-(u^{\prime}(1))^3+3u^{\prime}(1)=0.
\end{equation}
From Theorem \ref{thm-ns}, we can directly deduce
\begin{corollary}\label{cor1}
 The exterior Dirichlet problem \eqref{sleq1}, \eqref{sleq2} has a unique locally convex function $u\in \Phi_3$ satisfying
$$u(x)=\dfrac{\sqrt{3}}{2}|x|^2+c+O(|x|^{-1}),\ |x|\to\infty$$
if and only if $c\in[\mu(\alpha_1),+\infty)$, where
$$c=\mu(\alpha)=b-\dfrac{\sqrt{3}}{2}+\sqrt{3}\dint_1^{\infty}s\left[\left(1+\dfrac{\alpha}{3s^{2}\tilde{U}(s,\alpha)}\right)^{\frac{1}{2}}-1\right]ds,$$
$$\alpha_1:=\disp\sup_{r\geq 1}\tilde{U}^{-1}_r(1),$$
and $\tilde{U}=\tilde{U}(r,\alpha)=\tilde{U}_r(\alpha)$ satisfies
$$\tilde{U}^3-3r^2\tilde{U}-\alpha=0.$$
\end{corollary}

\begin{proof}[Proof of Theorem \ref{hqn2}] Let $u(x)=u(|x|)=u(r)$ be the radially symmetric solution of the problem \eqref{sl2eq1} and \eqref{sl2eq2}. Then
\begin{equation}\label{s2rad}
u^{\prime\prime}\dfrac{u^{\prime}}{r}=u^{\prime\prime}+\dfrac{u^{\prime}}{r},
\end{equation}
i.e.,
$$((u^{\prime})^2)^{\prime}=2(ru^{\prime})^{\prime}.$$
Integrating on both sides from $1$ to $r$, we have
\begin{equation}\label{u21}(u^{\prime})^2-2ru^{\prime}-\rho=0,\end{equation}
where $\rho=(u^{\prime}(1))^2-2u^{\prime}(1)=(u^{\prime}(1)-1)^2-1.$
By Lemma \ref{lemma1}, we know that for $k=n=m=2,l=1$,
$$\dfrac{u^{\prime}}{r}>\left(\dfrac{lC_n^l}{kC_n^k}\right)^{\frac{1}{k-l}}=1,$$
so $u^{\prime}(r)>r$ for any $r\geq 1.$ Therefore \eqref{u21} has a solution
$$u^{\prime}(r)=r+\sqrt{r^2+\rho}$$
if and only if $r^2+\rho\geq 0$, i.e., $\rho\geq -1.$ Integrate on both sides from $1$ to $r$,
$$u(r)=b-\dfrac{1}{2}+\dfrac{1}{2}r^2+\dfrac{1}{2}[r\sqrt{r^2+\rho}+\rho \ln(r+\sqrt{r^2+\rho})]-\dfrac{1}{2}[\sqrt{1+\rho}+\rho \ln(1+\sqrt{1+\rho})].$$
Since
$$r\sqrt{r^2+\rho}=r^2+\dfrac{\rho}{2}+O(r^{-2}),\ r\to\infty,$$
and
$$\ln(r+\sqrt{r^2+\rho})=\ln r+\ln 2+O(r^{-2}),\ r\to\infty,$$
then
$$u(r)=r^2+\dfrac{\rho}{2}\ln r+\nu(\rho)+O(r^{-2}),$$
where
$$\nu(\rho)=b-\dfrac{1}{2}+\dfrac{\rho}{4}+\dfrac{\rho}{2}\ln 2-\dfrac{1}{2}[\sqrt{1+\rho}+\rho \ln(1+\sqrt{1+\rho})].$$
Theorem \ref{hqn2} is proved.
\end{proof}
Similar to \cite{WB}, $\nu(\rho)$ increases in $[-1,0]$ and decreases in $[0,+\infty)$. So
$$\nu(\rho)\leq \nu(0)=b-1,\ \rho\geq -1.$$
\begin{corollary}
The exterior Dirichlet problem \eqref{sl2eq1}, \eqref{sl2eq2} has a unique solution $u\in \Phi_2$ satisfying \eqref{u22} if and only if $c\leq b-1.$
\end{corollary}

\begin{remark}
If $k=n=2,l=0,$ we can refer to Theorem 2 in \cite{WB}.
\end{remark}

\section{Generalized symmetric functions, subsolutions and supersolutions}

In this section, we will construct a family of generalized symmetric smooth subsolutions and supersolutions of \eqref{hq1}. Now, we first give the definitions of generalized symmetric functions and generalized symmetric solutions, see \cite{BLL}.

\begin{definition}\label{Definition}
Let the diagonal matrix $A=\mbox{diag}(a_{1},\dots,a_{n})$, a function $u$ is called a generalized symmetric function with respect to $A$ if $u$ is a function of $$r:=r_A(x):=\sqrt{x^{T}Ax}=\left(\sum \limits_{i=1}^{n}a_{i}x_{i}^{2}\right)^{\frac{1}{2}}.$$

If $u$ is both a generalized symmetric function with respect to $A$ and a subsolution (supersolution, solution) of \eqref{hq1}, then we say that $u$ is a generalized symmetric subsolution (supersolution, solution) of \eqref{hq1}.
\end{definition}

Since
$$r^2=x^{T}Ax=\sum_{i=1}^na_ix_i^2,$$
then
$$2r\partial_{x_i}r=\partial_{x_i}(r^2)=2a_ix_i\ \ \mbox{and}\ \ \partial_{x_i}r=\frac{a_ix_i}{r}.$$
So $\tilde{\Phi}(x):=\tilde{\phi}(r)$ with $\tilde{\phi}\in C^2[0,+\infty)$ satisfies that
$$\partial_{x_i}\tilde{\Phi}(x)=\tilde{\phi}'(r)\partial_{x_i}r=\frac{\tilde{\phi}'(r)}{r}a_ix_i,$$
\begin{align*}
\partial_{x_ix_j}\tilde{\Phi}(x)=&\frac{\tilde{\phi}'(r)}{r}a_i\delta_{ij}+\frac{\tilde{\phi}''(r)-\frac{\tilde{\phi}'(r)}{r}}{r^2}(a_ix_i)(a_jx_j)\\
=&\tilde{h}(r)a_i\delta_{ij}+\frac{\tilde{h}'(r)}{r}(a_ix_i)(a_jx_j),
\end{align*}
where $\tilde{h}(r)=\tilde{\phi}^{\prime}(r)/r$.
Consequently,
$$D^2\tilde{\Phi}=\left(\tilde{h}(r)a_i\delta_{ij}+\frac{\tilde{h}'(r)}{r}(a_ix_i)(a_jx_j)\right)_{n\times n}.$$
Therefore by Lemma \ref{A1}, we can get that
\begin{align}\label{Skl}
&S_k(D^2\tilde{\Phi})=\sigma_k(\lambda(D^2\tilde{\Phi}))\nonumber\\
=&\sigma_k(a)\tilde{h}(r)^k+\frac{\tilde{h}'(r)}{r}\tilde{h}(r)^{k-1}\sum_{i=1}^{n}\sigma_{k-1;i}(a)a_i^2x_i^2.
\end{align}

In this paper, we always assume that $A$ is diagonal but not $A=\hat{a}I,\hat{a}=(C_n^l/C_n^k)^{\frac{1}{k-l}}$ because the Hessian quotient equation is not invariant under affine transformation. Detailed arguments for this can be referred to \cite{BLL}.
Let the matrix $A=\mbox{diag}(a_{1},\dots,a_{n})$ and the vector $a=(a_{1},\dots,a_{n})$. If $A\in \mathcal{A}_{k,l}$, then we have $a_{i}>0 (i=1,\dots,n)$ and $\sigma_{k}(a)/ \sigma_{l}(a)= 1$. For any fixed $t-$tuple $\{i_{1},\dots,i_{t}\}\subset\{1,\dots,n\}, 1\leq t\leq n-k$, let
$$\sigma_{k;i_{1}\cdots i_{t}}(a)=\sigma_{k}(a)|_{a_{i1}=\cdots=a_{it}=0},$$
that is, $\sigma_{k;i_{1}\dots i_{t}}(a)$ is the $k-$th order elementary symmetric function of the $n-t$ variables
$\{a_{i}|i\in\{1,\dots,n\}\backslash\{i_{1},\dots,i_{t}\}\}$.


Let $a=(a_{1},\dots,a_{n})\in \mathbb{R}^n\b\{0\}$, then by the definitions of $\Lambda_k$ and $\o{t}_k$,
$$\Lambda_k:=\Lambda_k(a,x):=\dfrac{\sum_{i=1}^n\sigma_{k-1;i}(a)a_i^2x_i^2}{\sigma_k(a)\sum_{i=1}^na_ix_i^2},\ \mbox{for any}\ x\in \mathbb{R}^n\b\{0\},$$
$$
\o{t}_{k}:=\o{t}_{k}(a):=\sup\limits_{x\in\mathbb{R}^n\b\{0\}}\Lambda_k(a,x),$$
and
$$\u{t}_{k}:=\u{t}_{k}(a):=\inf\limits_{x\in\mathbb{R}^n\b\{0\}}\Lambda_k(a,x).$$
\begin{lemma}\label{lem5} Let the vector $a=(a_1,a_2,\dots,a_n)$ satisfy $0<a_1\leq a_2\leq \dots\leq a_n.$ Then
for $1\leq k\leq n$,
\begin{equation}\label{thetak1}0<\u{t}_k\leq\frac{k}{n}\leq \o{t}_{k}\leq 1,\end{equation}
$$0=\o{t}_0<\frac{1}{n}\leq \frac{a_n}{\sigma_1(a)}=\o{t}_1\leq \o{t}_2\leq \dots\leq \o{t}_{n-1}<\o{t}_n=1,$$
and
$$0=\u{t}_0<\frac{a_1}{\sigma_1(a)}=\u{t}_1\leq \u{t}_2\leq \dots\leq \u{t}_{n-1}<\u{t}_n=1.$$
Moreover, for $1\leq k\leq n-1,$
$$\u{t}_k=\o{t}_k=\frac{k}{n}$$
if and only if $a_{1}=\cdots=a_{n}=\o{C}$ for some $\o{C}>0$.\end{lemma}

\begin{proof}
See \cite{LLX1}.
\end{proof}

\begin{remark}\label{rem1}
From \eqref{thetak1}, we know that
$$0<\frac{a_1}{\sigma_1(a)}\leq \u{t}_l\leq \frac{l}{n}<\frac{k}{n}\leq \o{t}_k\leq 1.$$
Then
$$0<\o{t}_k-\u{t}_l<1.$$
\end{remark}

Let $C_1$ be a positive constant and $\theta_0$ be sufficiently large. Assume that $\overline{g},\underline{g}\in C^{0}([0,+\infty))$ are positive functions of $r=\sqrt{x^{T}Ax}$  satisfying
$$0<\inf_{r\in[0,+\infty)}\u{g}(r)\leq \underline{g}(r)\leq g(x)\leq \overline{g}(r)\leq \sup_{r\in[0,+\infty)}\o{g}(r)<+\infty,\ x\in \mathbb{R}^n,$$
\begin{equation}\label{three}\underline{g}(r)\leq g_0(r)\leq \overline{g}(r),\ x\in \mathbb{R}^n,\end{equation}
$\u{g}(r)$ is strictly increasing in $r$ and for $\beta>2,$
\begin{equation}\label{2} \overline{g}(r)= g_0(r)+C_{1}r^{-\beta},~r> \theta_0,
 \end{equation}
$$\underline{g}(r)=g_0(r)-C_{1}r^{-\beta},~r>\theta_0.
$$

To construct the subsolutions of \eqref{hq1}, we need to seek the solutions or subsolutions in $\Gamma_k$ with
appropriate properties of
\begin{equation}\label{super-g}
\frac{S_{k}(D^{2}v)}{S_{l}(D^{2}v)}=\overline{g}.
\end{equation}
These solutions or subsolutions are apparently subsolutions of \eqref{hq1}. However, from the following Proposition \ref{prop1}, we can get that there does not exist generalized symmetric solution of \eqref{hq1} for $1\leq k\leq n-1$ unless $A=\hat{a}I$.

\begin{prop}\label{prop1}
Let $A={\rm{diag}}(a_{1},\dots,a_{n})\in \mathcal{A}_{k,l},0\leq l<k\leq n,$ and $0<r_1<r_2<\infty.$ If there exists a function $G\in C^2(r_1,r_2)$ such that $T(x)=G(r)=G(\sqrt{x^{T}Ax})$ is a generalized symmetric solution of \eqref{super-g}, then
$$k=n,\ \ \mbox{or}\ \ a_1=\cdots=a_n=\hat{a}=(C_n^l/C_n^k)^{\frac{1}{k-l}}.$$
\end{prop}

\begin{proof}
For the special case $l=0,1\leq k\leq n$, the Hessian equation case, Proposition \ref{prop1} can be proved similarly with Proposition 2.2 in \cite{CB}. We only need to prove the case $1\leq l<k\leq n$.

Let $J(r)=G^{\prime}(r)/r$. By \eqref{Skl}, we know that $T$ satisfies
\begin{align*}
\frac{S_k(D^2T)}{S_l(D^2T)}=\frac{\sigma_k(a)J(r)^k+\frac{J'(r)}{r}J(r)^{k-1}\displaystyle\sum_{i=1}^{n}\sigma_{k-1;i}(a)a_i^2x_i^2}
{\sigma_l(a)J(r)^l+\frac{J'(r)}{r}J(r)^{l-1}\displaystyle\sum_{i=1}^{n}\sigma_{l-1;i}(a)a_i^2x_i^2}=\o{g}(r).
\end{align*}
Set $x=(0,\dots,0,\sqrt{r/a_i},0,\dots,0)$. Then
\begin{align*}
\frac{S_k(D^2T)}{S_l(D^2T)}=\frac{\sigma_k(a)J(r)^k+J'(r)J(r)^{k-1}\sigma_{k-1;i}(a)a_i}
{\sigma_l(a)J(r)^l+J'(r)J(r)^{l-1}\sigma_{l-1;i}(a)a_i}=\o{g}(r).
\end{align*}
So
\begin{equation}\label{equ2}\sigma_k(a)J(r)^k+J'(r)J(r)^{k-1}\sigma_{k-1;i}(a)a_i=\o{g}(r)[\sigma_l(a)J(r)^l+J'(r)J(r)^{l-1}\sigma_{l-1;i}(a)a_i].
\end{equation}
Since $\sigma_k(a)=\sigma_l(a)$, then
\begin{equation}\label{eq1}
\frac{J(r)^k-\o{g}(r)J(r)^l}{J^{\prime}(r)}=\frac{\o{g}(r)J(r)^{l-1}\sigma_{l-1;i}(a)a_i-J(r)^{k-1}\sigma_{k-1;i}(a)a_i}{\sigma_k(a)}.\end{equation}
Noting that the left side of \eqref{eq1} is independent of $i$, so for any $i\not=j,$ we have that
$$\o{g}(r)J(r)^{l-1}\sigma_{l-1;i}(a)a_i-J(r)^{k-1}\sigma_{k-1;i}(a)a_i=\o{g}(r)J(r)^{l-1}\sigma_{l-1;j}(a)a_j-J(r)^{k-1}\sigma_{k-1;j}(a)a_j.$$
As a result
\begin{equation}\label{eq2}\o{g}(r)J(r)^{l-1}[\sigma_{l-1;i}(a)a_i-\sigma_{l-1;j}(a)a_j]=J(r)^{k-1}[\sigma_{k-1;i}(a)a_i-\sigma_{k-1;j}(a)a_j].
\end{equation}
Applying the equality $\sigma_k(a)=\sigma_{k;i}(a)+a_i\sigma_{k-1;i}(a)$ for all $i,$ we get that
\begin{align*}&\sigma_{l-1;i}(a)a_i-\sigma_{l-1;j}(a)a_j\\
=&[\sigma_{l-1;ij}(a)+\sigma_{l-2;ij}(a)a_j]a_i-[\sigma_{l-1;ij}(a)+\sigma_{l-2;ij}(a)a_i]a_j\\
=&\sigma_{l-1;ij}(a)(a_i-a_j).
\end{align*}
Therefore \eqref{eq2} becomes
$$\o{g}(r)J(r)^{l-1}\sigma_{l-1;ij}(a)(a_i-a_j)=J(r)^{k-1}\sigma_{k-1;ij}(a)(a_i-a_j).$$
If $k=n$, then $\sigma_{k-1;ij}(a)=0$ for any $i\not=j$. But $\sigma_{l-1;ij}(a)>0$, so we get that
$$a_1=\dots=a_n=\hat{a}.$$
If $1\leq l<k\leq n-1$, then $\sigma_{k-1;ij}(a)>0$ and $\sigma_{l-1;ij}(a)>0$. Suppose on the contrary that $a_i\not=a_j$, then
$$\o{g}(r)J(r)^{l-1}\sigma_{l-1;ij}(a)=J(r)^{k-1}\sigma_{k-1;ij}(a).$$
Thus
\begin{equation}\label{equ3}\frac{\sigma_{k-1;ij}(a)}{\sigma_{l-1;ij}(a)}=\frac{\o{g}(r)J(r)^{l-1}}{J(r)^{k-1}}=\o{g}(r)J(r)^{l-k}.\end{equation}
Since the left side is independent of $r$, then $\o{g}(r)J(r)^{l-k}$ is a constant $c_0>0.$ So $\o{g}(r)=c_0J(r)^{k-l}.$ Substituting into \eqref{eq1}, we have that
\begin{equation}\label{equ4}\frac{J(r)(1-c_0)}{c_0J^{\prime}(r)}=\frac{\sigma_{l-1;i}(a)a_i-\sigma_{k-1;i}(a)a_i}{\sigma_k(a)}.\end{equation}
Since the left side of the above equality is independent of $i$, then for any $i\not=j$,
$$\sigma_{l-1;i}(a)a_i-\sigma_{k-1;i}(a)a_i=\sigma_{l-1;j}(a)a_j-\sigma_{k-1;j}(a)a_j,$$
so
$$\sigma_{l-1;ij}(a)(a_i-a_j)=\sigma_{k-1;ij}(a)(a_i-a_j).$$
However $a_i\not=a_j$, thus $\sigma_{l-1;ij}(a)=\sigma_{k-1;ij}(a)$. Therefore by \eqref{equ3}, we can have $c_0=1$. Then by \eqref{equ4}, we can get that for all $i$,
$$\sigma_{l-1;i}(a)a_i=\sigma_{k-1;i}(a)a_i.$$
Recalling the equality
$$\sum_{i=1}^{n}a_i\sigma_{k-1;i}(a)=k\sigma_k(a),$$
we know that $k\sigma_k(a)=l\sigma_l(a)$. Since $A\in \mathcal{A}_{k,l}$, then $\sigma_k(a)=\sigma_l(a)$, so $k=l$. This is a contradiction.
\end{proof}

So in virtue of Proposition \ref{prop1}, we can only find the generalized subsolutions of \eqref{hq1}. We will construct the generalized symmetric subsolution $W(x)=w(r)$ of \eqref{hq1}. First we discuss the function $h(r)$ which actually equals to $w^{\prime}(r)/r$.

\begin{lemma}\label{lem2}
Let $0\leq l<k\leq n, n\geq 3,A\in \mathcal{A}_{k,l}, a:=(a_1,a_2,\dots,a_n):=\lambda(A),0<a_1\leq a_2\leq \dots\leq a_n$ and $\delta>\sup_{r\in[1,+\infty)} \o{g}^{\frac{1}{k-l}}(r).$ Then the problem
\begin{equation}\label{w}
\begin{cases}
\dfrac{h(r)^k+\o{t}_k r h(r)^{k-1}h'(r)}{h(r)^l+\u{t}_l r h(r)^{l-1}h'(r)}=\overline{g}(r),\ r>1,\\
h(1)=\delta,\\
h(r)^k+\o{t}_k r h(r)^{k-1}h'(r)>0,
\end{cases}
\end{equation}
has a smooth solution $h(r)=h(r,\delta)$ on $[1,+\infty)$ satisfying

{\rm(i)}~ $\o{g}^{\frac{1}{k-l}}(r)\leq h(r,\delta)\leq \delta,$ and $\partial_{r}h(r,\delta)\leq 0$.

{\rm(ii)}~$h(r,\delta)$ is continuous and strictly increasing in $\delta$ and
$$\lim_{\delta\to+\infty}h(r,\delta)=+\infty,\ \forall\ r\geq 1.$$


\end{lemma}

\begin{proof}
For brevity, sometimes we write $h(r)$ or $h(r,\delta)$, $\o{t}_k(a)$ or $\o{t}_k$ and $\u{t}_l(a)$ or $\u{t}_l$, when there is no confusion. Due to \eqref{w}, we have
\begin{equation}\label{w1}\begin{cases}\displaystyle\frac{\mbox{d} h}{\mbox{d}r}=-\frac{1}{r}\dfrac{h}{\o{t}_k}\dfrac{h(r)^{k-l}-\o{g}(r)}{h(r)^{k-l}-\o{g}(r)\frac{\u{t}_l}{\o{t}_k}},\ r>1,\\
h(1)=\delta.
\end{cases}
\end{equation}
Since $\delta> \sup_{r\in[0,+\infty)} \o{g}^{\frac{1}{k-l}}(r)$ and $\u{t}_l/\o{t}_k<1$, then by the existence, uniqueness and extension theorem for the solution of the initial value problem of the ODE, we obtain that the problem \eqref{w} has a smooth solution $h(r,\delta)$ such that $\o{g}^{\frac{1}{k-l}}(r)\leq h(r,\delta)\leq \delta,$ and $\partial_{r}h(r,\delta)\leq 0$. Then assertion (i) of the lemma is proved. In addition, the existence of solution can be proved by the following.

By \eqref{w}, we can know that
$$\dfrac{\frac{k}{\o{t}_k}r^{\frac{k}{\o{t}_k}-1}(h(r)^k+\o{t}_k r h(r)^{k-1}h'(r))}{\frac{l}{\u{t}_l}r^{\frac{l}{\u{t}_l}-1}(h(r)^l+\u{t}_l r h(r)^{l-1}h'(r))}=\dfrac{\frac{k}{\o{t}_k}r^{\frac{k}{\o{t}_k}-1}}{\frac{l}{\u{t}_l}r^{\frac{l}{\u{t}_l}-1}}\overline{g}(r),\ r>1.$$
That is
$$\dfrac{(r^{\frac{k}{\o{t}_k}}h^{k}(r))^{\prime}}{(r^{\frac{l}{\u{t}_l}}h^{l}(r))^{\prime}}
=\dfrac{k\u{t}_l}{l\o{t}_k}r^{\frac{k}{\o{t}_k}-\frac{l}{\u{t}_l}}\o{g}(r).$$
So
$$(r^{\frac{k}{\o{t}_k}}h^{k}(r))^{\prime}=\dfrac{k\u{t}_l}{l\o{t}_k}r^{\frac{k}{\o{t}_k}-\frac{l}{\u{t}_l}}\o{g}(r)(r^{\frac{l}{\u{t}_l}}h^{l}(r))^{\prime}.$$
Integrating the above equality from $1$ to $r$, we have that
$$\dint_1^r(s^{\frac{k}{\o{t}_k}}h^{k}(s))^{\prime}ds
=\dfrac{k\u{t}_l}{l\o{t}_k}\dint_1^rs^{\frac{k}{\o{t}_k}-\frac{l}{\u{t}_l}}\o{g}(s)(s^{\frac{l}{\u{t}_l}}h^{l}(s))^{\prime}ds.$$
Then
\begin{equation}\label{w4}r^{\frac{k}{\o{t}_k}}h^{k}(r)-\delta^k
=\dfrac{k\u{t}_l}{l\o{t}_k}\dint_1^rs^{\frac{k}{\o{t}_k}-\frac{l}{\u{t}_l}}\o{g}(s)(s^{\frac{l}{\u{t}_l}}h^{l}(s))^{\prime}ds.
\end{equation}
Since $h(r)^k+\o{t}_k r h(r)^{k-1}h'(r)>0$, then $h(r)^l+\u{t}_l r h(r)^{l-1}h'(r)>0$, so also $(r^{\frac{l}{\u{t}_l}}h^{l}(r))^{\prime}>0.$  According to the mean value theorem of integrals, we have that there exists $1\leq \theta_1\leq r$ such that
$$r^{\frac{k}{\o{t}_k}}h^{k}(r)-\delta^k
=\dfrac{k\u{t}_l}{l\o{t}_k}\theta_1^{\frac{k}{\o{t}_k}-\frac{l}{\u{t}_l}}\o{g}(\theta_1)\dint_1^r(s^{\frac{l}{\u{t}_l}}h^{l}(s))^{\prime}ds,$$
i.e.,
$$r^{\frac{k}{\o{t}_k}}h^{k}(r)-\delta^k
=\dfrac{k\u{t}_l}{l\o{t}_k}\theta_1^{\frac{k}{\o{t}_k}-\frac{l}{\u{t}_l}}\o{g}(\theta_1)[r^{\frac{l}{\u{t}_l}}h^{l}(r)-\delta^l].$$
As a result,
\begin{equation}\label{w2}r^{\frac{k}{\o{t}_k}}h^{k}(r)-\delta^k
-\dfrac{k\u{t}_l}{l\o{t}_k}\theta_1^{\frac{k}{\o{t}_k}-\frac{l}{\u{t}_l}}\o{g}(\theta_1)r^{\frac{l}{\u{t}_l}}h^{l}(r)
+\dfrac{k\u{t}_l}{l\o{t}_k}\theta_1^{\frac{k}{\o{t}_k}-\frac{l}{\u{t}_l}}\o{g}(\theta_1)\delta^l=0.\end{equation}
Denote the left side of the above equality as $F(h,\delta,\theta_1,r)$, that is,
$$F(h,\delta,\theta_1,r)=r^{\frac{k}{\o{t}_k}}h^{k}(r)-\delta^k
-\dfrac{k\u{t}_l}{l\o{t}_k}\theta_1^{\frac{k}{\o{t}_k}-\frac{l}{\u{t}_l}}\o{g}(\theta_1)r^{\frac{l}{\u{t}_l}}h^{l}(r)
+\dfrac{k\u{t}_l}{l\o{t}_k}\theta_1^{\frac{k}{\o{t}_k}-\frac{l}{\u{t}_l}}\o{g}(\theta_1)\delta^l.$$
Then we assert
\begin{equation}\label{partialF}\dfrac{\partial F}{\partial h}
=k\left[r^{\frac{k}{\o{t}_k}}h^{k-1}
-\dfrac{\u{t}_l}{\o{t}_k}\theta_1^{\frac{k}{\o{t}_k}-\frac{l}{\u{t}_l}}\o{g}(\theta_1)r^{\frac{l}{\u{t}_l}}h^{l-1}\right]>0.\end{equation}

In fact, by \eqref{w2}, we can get that
\begin{align}r^{\frac{k}{\o{t}_k}}h^{k}(r)
-\dfrac{k\u{t}_l}{l\o{t}_k}\theta_1^{\frac{k}{\o{t}_k}-\frac{l}{\u{t}_l}}\o{g}(\theta_1)r^{\frac{l}{\u{t}_l}}h^{l}(r)
=&\delta^k-\dfrac{k\u{t}_l}{l\o{t}_k}\theta_1^{\frac{k}{\o{t}_k}-\frac{l}{\u{t}_l}}\o{g}(\theta_1)\delta^l\nonumber\\
=&\delta^l\left[\delta^{k-l}-\dfrac{k\u{t}_l}{l\o{t}_k}\theta_1^{\frac{k}{\o{t}_k}-\frac{l}{\u{t}_l}}\o{g}(\theta_1)\right]\label{w3}.
\end{align}
Since
$$\u{t}_l\leq \frac{l}{n}<\frac{k}{n}\leq \o{t}_k,$$
then $$\dfrac{k\u{t}_l}{l\o{t}_k}\leq 1\ \ \mbox{and}\ \ \ \frac{k}{\o{t}_k}-\frac{l}{\u{t}_l}\leq 0.$$
Combining with $\delta^{k-l}>\sup_{r\in [0,+\infty)}\o{g}(r)$, we have that
\begin{equation}\label{equ1}\delta^{k-l}-\dfrac{k\u{t}_l}{l\o{t}_k}\theta_1^{\frac{k}{\o{t}_k}-\frac{l}{\u{t}_l}}\o{g}(\theta_1)>0.\end{equation}
Therefore by \eqref{w3}, $$r^{\frac{k}{\o{t}_k}}h^{k}(r)
-\dfrac{k\u{t}_l}{l\o{t}_k}\theta_1^{\frac{k}{\o{t}_k}-\frac{l}{\u{t}_l}}\o{g}(\theta_1)r^{\frac{l}{\u{t}_l}}h^{l}(r)>0.$$
However,
$$-\dfrac{\u{t}_l}{\o{t}_k}>-\dfrac{k\u{t}_l}{l\o{t}_k},$$
then we conclude that
\begin{align*}0<&r^{\frac{k}{\o{t}_k}}h^{k}(r)
-\dfrac{k\u{t}_l}{l\o{t}_k}\theta_1^{\frac{k}{\o{t}_k}-\frac{l}{\u{t}_l}}\o{g}(\theta_1)r^{\frac{l}{\u{t}_l}}h^{l}(r)\\
<&r^{\frac{k}{\o{t}_k}}h^{k}(r)
-\dfrac{\u{t}_l}{\o{t}_k}\theta_1^{\frac{k}{\o{t}_k}-\frac{l}{\u{t}_l}}\o{g}(\theta_1)r^{\frac{l}{\u{t}_l}}h^{l}(r).\end{align*}
Thus
$$\dfrac{\partial F}{\partial h}>0.$$

By the implicit function theorem, \eqref{w2} can determine a unique function $h(r):=h(r,\delta,\theta_1)$. Moreover,
$$\frac{\partial h}{\partial\delta}=-\frac{\partial F}{\partial \delta}/\frac{\partial F}{\partial h}.$$
So due to \eqref{partialF} and \eqref{equ1},
\begin{align}
\frac{\partial h}{\partial \delta}=&\frac{k\delta^{l-1}\left[\delta^{k-l}-\frac{\u{t}_l}{\o{t}_k}\theta_1^{\frac{k}{\o{t}_k}-\frac{l}{\u{t}_l}}\o{g}(\theta_1)\right]}
{kr^{\frac{l}{\u{t}_l}}h^{l-1}[r^{\frac{k}{\o{t}_k}-\frac{l}{\u{t}_l}}h^{k-l}-\frac{\u{t}_l}{\o{t}_k}\theta_1^{\frac{k}{\o{t}_k}-\frac{l}{\u{t}_l}}\o{g}(\theta_1)]}
\nonumber\\
>&0\label{walpha}.
\end{align}
In virtue of \eqref{w3}, we deduce that
\begin{align*}h^{l}(r)\left[r^{\frac{k}{\o{t}_k}}h^{k-l}(r)
-\dfrac{k\u{t}_l}{l\o{t}_k}\theta_1^{\frac{k}{\o{t}_k}-\frac{l}{\u{t}_l}}\o{g}(\theta_1)r^{\frac{l}{\u{t}_l}}\right]
=\delta^l\left[\delta^{k-l}-\dfrac{k\u{t}_l}{l\o{t}_k}\theta_1^{\frac{k}{\o{t}_k}-\frac{l}{\u{t}_l}}\o{g}(\theta_1)\right].
\end{align*}
If $\delta\to+\infty,$ then the right side of the above equality tends to $+\infty$. Since $h$ is increasing in $\delta$ by \eqref{walpha}, then
$$h(r,\delta,\theta_1)\to+\infty,\ \ \mbox{as}\ \ \delta\to+\infty.$$ The lemma is proved.
\end{proof}

\begin{remark}\label{remsub1}
By the extension theorem of solutions, we can know that the solution $h$ in \eqref{w} can be extended to the left of $[1,+\infty)$ and then is well defined in $[0,+\infty).$
\end{remark}

Let the function $h_0$ satisfy
\begin{equation}\label{h0}
\begin{cases}
\dfrac{h_0(r)^k+\o{t}_k r h_0(r)^{k-1}h_0'(r)}{h_0(r)^l+\u{t}_l r h_0(r)^{l-1}h_0'(r)}=g_0(r),\ r>0,\\
h_0(0)>\sup_{r\in [0,+\infty)} g_0^{\frac{1}{k-l}}(r),\\
h_0(r)^k+\o{t}_k r h_0(r)^{k-1}h_0'(r)>0.
\end{cases}
\end{equation}
Alike Lemma \ref{lem2}, we know that $h_0=h_0(r)\in  C^0[0,+\infty)\cap C^1(0,+\infty)$ is bounded. From the equation in \eqref{h0}, we have that
$$\dfrac{\frac{k-l}{\o{t}_k-\u{t}_l}r^{\frac{k-l}{\o{t}_k-\u{t}_l}-1}h_0(r)^{\frac{k\u{t}_l-l\o{t}_k}{\o{t}_k-\u{t}_l}}[h_0(r)^k+\o{t}_k r h_0(r)^{k-1}h_0'(r)]}{\frac{k-l}{\o{t}_k-\u{t}_l}r^{\frac{k-l}{\o{t}_k-\u{t}_l}-1}h_0(r)^{\frac{k\u{t}_l-l\o{t}_k}{\o{t}_k-\u{t}_l}}[h_0(r)^l+\u{t}_l r h_0(r)^{l-1}h_0'(r)]}=g_0(r),\ r>0.$$
So
\begin{equation}\label{*1}\dfrac{(r^{\frac{k-l}{\o{t}_k-\u{t}_l}}h_0^{\frac{(k-l)\o{t}_k}{\o{t}_k-\u{t}_l}})^{\prime}}
{(r^{\frac{k-l}{\o{t}_k-\u{t}_l}}h_0^{\frac{(k-l)\u{t}_l}{\o{t}_k-\u{t}_l}})^{\prime}}=g_0(r),r>0.
\end{equation}
Integrating the above equality on both sides, we know that $h_0$ satisfies
\begin{equation}\label{h_0}h_0(r)=\left(r^{-\frac{k-l}{\o{t}_k-\u{t}_l}}
\dint_0^r g_0(s)(s^{\frac{k-l}{\o{t}_k-\u{t}_l}}h_0(s)^{\frac{(k-l)\u{t}_l}{\o{t}_k-\u{t}_l}})^{\prime}ds\right)^{\frac{\o{t}_k-\u{t}_l}{(k-l)\o{t}_k}}.\end{equation}

For $A\in \mathcal{A}_{k,l}$ and $\gamma>0$, let
$$E_{\gamma}:=\{x\in \mathbb{R}^n|x^{T}Ax<\gamma^2\}=\{x\in \mathbb{R}^n|r_A(x)<\gamma\},$$
where $r_{A}(x):=\sqrt{x^{T}Ax}$. For some constant $\beta_1$, define
$$w(r):=w_{\beta_1,\eta,\delta}(r):=\beta_1+\int_{\eta}^{r}\theta h(\theta,\delta)d\theta,\ \forall \ r\geq \eta>1. $$
Then
\begin{align}
w(r)
=&\beta_1+\int_{\eta}^{+\infty}\theta h(\theta)d\theta-\int_{r}^{+\infty}\theta h(\theta)d\theta\nonumber\\
=&\beta_1+\int_{\eta}^{+\infty}\theta h(\theta)d\theta+\int_{0}^r\theta h_0(\theta)d\theta-\int_{0}^r\theta h_0(\theta)d\theta
-\int_{r}^{+\infty}\theta h(\theta)d\theta\nonumber\\
=&\beta_1+\int_{\eta}^{+\infty}\theta h(\theta)d\theta+\int_{0}^r\theta h_0(\theta)d\theta-\int_{0}^\eta\theta h_0(\theta)d\theta-\int_\eta^{+\infty}\theta h_0(\theta)d\theta\nonumber\\
&+\int_{r}^{+\infty}\theta h_0(\theta)d\theta-\int_{r}^{+\infty}\theta h(\theta)d\theta\nonumber\\
=&\int_{0}^r\theta h_0(\theta)d\theta+\beta_1-\int_{0}^\eta \theta h_0(\theta)d\theta+\int_{\eta}^{+\infty}\theta[h(\theta)-h_0(\theta)]d\theta\nonumber\\
&-\int_{r}^{+\infty}\theta[h(\theta)-h_0(\theta)]d\theta\nonumber\\
=&\int_{0}^r\theta h_0(\theta)d\theta+\mu_{\beta_1,\eta}(\delta)-\int_{r}^{+\infty}\theta[h(\theta)-h_0(\theta)]d\theta\label{v},
\end{align}
where
$$\mu_{\beta_1,\eta}(\delta):=\beta_1-\int_{0}^\eta \theta h_0(\theta)d\theta+\int_{\eta}^{+\infty}\theta[h(\theta)-h_0(\theta)]d\theta.$$

From \eqref{w},
$$\dfrac{\frac{k-l}{\o{t}_k-\u{t}_l}r^{\frac{k-l}{\o{t}_k-\u{t}_l}-1}h(r)^{\frac{k\u{t}_l-l\o{t}_k}{\o{t}_k-\u{t}_l}}[h(r)^k+\o{t}_k r h(r)^{k-1}h'(r)]}{\frac{k-l}{\o{t}_k-\u{t}_l}r^{\frac{k-l}{\o{t}_k-\u{t}_l}-1}h(r)^{\frac{k\u{t}_l-l\o{t}_k}{\o{t}_k-\u{t}_l}}[h(r)^l+\u{t}_l r h(r)^{l-1}h'(r)]}=\overline{g}(r),\ r>1.$$
So
\begin{equation}\label{*2}\dfrac{(r^{\frac{k-l}{\o{t}_k-\u{t}_l}}h^{\frac{(k-l)\o{t}_k}{\o{t}_k-\u{t}_l}})^{\prime}}
{(r^{\frac{k-l}{\o{t}_k-\u{t}_l}}h^{\frac{(k-l)\u{t}_l}{\o{t}_k-\u{t}_l}})^{\prime}}=\o{g}(r),r>1,\end{equation}
that is
$$(r^{\frac{k-l}{\o{t}_k-\u{t}_l}}h^{\frac{(k-l)\o{t}_k}{\o{t}_k-\u{t}_l}})^{\prime}
=\o{g}(r)(r^{\frac{k-l}{\o{t}_k-\u{t}_l}}h^{\frac{(k-l)\u{t}_l}{\o{t}_k-\u{t}_l}})^{\prime},r>1.$$
Integrating the above equality from $1$ to $r$, we obtain that
$$r^{\frac{k-l}{\o{t}_k-\u{t}_l}}h^{\frac{(k-l)\o{t}_k}{\o{t}_k-\u{t}_l}}-\delta^{\frac{(k-l)\o{t}_k}{\o{t}_k-\u{t}_l}}
=\dint_1^r \o{g}(s)(s^{\frac{k-l}{\o{t}_k-\u{t}_l}}h(s)^{\frac{(k-l)\u{t}_l}{\o{t}_k-\u{t}_l}})^{\prime}ds.$$
As a result,
\begin{equation}\label{h(r)}h(r)=\left(\delta^{\frac{(k-l)\o{t}_k}{\o{t}_k-\u{t}_l}}r^{-\frac{k-l}{\o{t}_k-\u{t}_l}}
+r^{-\frac{k-l}{\o{t}_k-\u{t}_l}}
\dint_1^r\o{g}(s)(s^{\frac{k-l}{\o{t}_k-\u{t}_l}}h(s)^{\frac{(k-l)\u{t}_l}{\o{t}_k-\u{t}_l}})^{\prime}ds\right)^{\frac{\o{t}_k-\u{t}_l}{(k-l)\o{t}_k}}.
\end{equation}
So
$$w(r)=\beta_1+\int_{\eta}^{r}\theta \left(\delta^{\frac{(k-l)\o{t}_k}{\o{t}_k-\u{t}_l}}\theta^{-\frac{k-l}{\o{t}_k-\u{t}_l}}
+\theta^{-\frac{k-l}{\o{t}_k-\u{t}_l}}
\dint_1^\theta\o{g}(s)(s^{\frac{k-l}{\o{t}_k-\u{t}_l}}h(s)^{\frac{(k-l)\u{t}_l}{\o{t}_k-\u{t}_l}})^{\prime}ds\right)^{\frac{\o{t}_k-\u{t}_l}{(k-l)\o{t}_k}}
d\theta,\ \forall \ r\geq \eta>1. $$

Since by \eqref{2}, $\o{g}(r)=g_0(r)+C_1r^{-\beta}, r>\theta_{0}$, then the last term in \eqref{v} is
\begin{align}\label{vinfty}
&-\int_{r}^{+\infty}\theta[h(\theta)-h_0(\theta)]d\theta\nonumber\\
=&-\int_{r}^{+\infty}\theta\left\{\left(\delta^{\frac{(k-l)\o{t}_k}{\o{t}_k-\u{t}_l}}\theta^{-\frac{k-l}{\o{t}_k-\u{t}_l}}
+\theta^{-\frac{k-l}{\o{t}_k-\u{t}_l}}
\dint_1^\theta\o{g}(s)(s^{\frac{k-l}{\o{t}_k-\u{t}_l}}h(s)^{\frac{(k-l)\u{t}_l}{\o{t}_k-\u{t}_l}})^{\prime}ds\right)^{\frac{\o{t}_k-\u{t}_l}{(k-l)\o{t}_k}}
-h_0(\theta)\right\}d\theta\nonumber\\
=&-\int_{r}^{+\infty}\theta\left\{\left[\delta_0 \theta^{-\frac{k-l}{\o{t}_k-\u{t}_l}}+\theta^{-\frac{k-l}{\o{t}_k-\u{t}_l}}\int_{\theta_{0}}^\theta
(g_0(s)+C_1s^{-\beta})(s^{\frac{k-l}{\o{t}_k-\u{t}_l}}h(s)^{\frac{(k-l)\u{t}_l}{\o{t}_k-\u{t}_l}})^{\prime}ds\right)^{\frac{\o{t}_k-\u{t}_l}{(k-l)\o{t}_k}}
-h_0(\theta)\right\}d\theta\nonumber\\
=&-\int_{r}^{+\infty}\theta\left\{\left[\delta_0 \theta^{-\frac{k-l}{\o{t}_k-\u{t}_l}}+\theta^{-\frac{k-l}{\o{t}_k-\u{t}_l}}\left(\int_{0}^\theta g_0(s)(s^{\frac{k-l}{\o{t}_k-\u{t}_l}}h(s)^{\frac{(k-l)\u{t}_l}{\o{t}_k-\u{t}_l}})^{\prime}ds
-\int_{0}^{\theta_{0}}g_0(s)(s^{\frac{k-l}{\o{t}_k-\u{t}_l}}h(s)^{\frac{(k-l)\u{t}_l}{\o{t}_k-\u{t}_l}})^{\prime}ds\right)\right.\right.\nonumber\\
&\left.\left.+\theta^{-\frac{k-l}{\o{t}_k-\u{t}_l}}\int_{\theta_{0}}^\theta C_1s^{-\beta}(s^{\frac{k-l}{\o{t}_k-\u{t}_l}}h(s)^{\frac{(k-l)\u{t}_l}{\o{t}_k-\u{t}_l}})^{\prime}ds\right]^{\frac{\o{t}_k-\u{t}_l}{(k-l)\o{t}_k}}
-h_0(\theta)\right\}d\theta,\nonumber\\
=&-\int_{r}^{+\infty}\theta\left\{\left[\delta_1 \theta^{-\frac{k-l}{\o{t}_k-\u{t}_l}}+\theta^{-\frac{k-l}{\o{t}_k-\u{t}_l}}\int_{0}^\theta g_0(s)(s^{\frac{k-l}{\o{t}_k-\u{t}_l}}h(s)^{\frac{(k-l)\u{t}_l}{\o{t}_k-\u{t}_l}})^{\prime}ds\right.\right.\nonumber\\
&\left.\left.
+\theta^{-\frac{k-l}{\o{t}_k-\u{t}_l}}\int_{\theta_{0}}^\theta C_1s^{-\beta}(s^{\frac{k-l}{\o{t}_k-\u{t}_l}}h(s)^{\frac{(k-l)\u{t}_l}{\o{t}_k-\u{t}_l}})^{\prime}ds\right]^{\frac{\o{t}_k-\u{t}_l}{(k-l)\o{t}_k}}
-h_0(\theta)\right\}d\theta\nonumber\\
=&-\int_{r}^{+\infty}\theta h_0(\theta)\left\{\left[\frac{\delta_1 \theta^{-\frac{k-l}{\o{t}_k-\u{t}_l}}}{h_0^{\frac{(k-l)\o{t}_k}{\o{t}_k-\u{t}_l}}(\theta)}+\frac{\theta^{-\frac{k-l}{\o{t}_k-\u{t}_l}}\int_{0}^\theta g_0(s)(s^{\frac{k-l}{\o{t}_k-\u{t}_l}}h(s)^{\frac{(k-l)\u{t}_l}{\o{t}_k-\u{t}_l}})^{\prime}ds
-h_0^{\frac{(k-l)\o{t}_k}{\o{t}_k-\u{t}_l}}(\theta)}{h_0^{\frac{(k-l)\o{t}_k}{\o{t}_k-\u{t}_l}}(\theta)}+1\right.\right.\nonumber\\
&\left.\left.
+\frac{\theta^{-\frac{k-l}{\o{t}_k-\u{t}_l}}\int_{\theta_{0}}^\theta C_1s^{-\beta}(s^{\frac{k-l}{\o{t}_k-\u{t}_l}}h(s)^{\frac{(k-l)\u{t}_l}{\o{t}_k-\u{t}_l}})^{\prime}ds}{h_0^{\frac{(k-l)\o{t}_k}{\o{t}_k-\u{t}_l}}(\theta)}
\right]^{\frac{\o{t}_k-\u{t}_l}{(k-l)\o{t}_k}}
-1\right\}d\theta,
\end{align}
where $\delta_0=\delta^{\frac{(k-l)\o{t}_k}{\o{t}_k-\u{t}_l}}+\int_1^{\theta_{0}}
\o{g}(s)(s^{\frac{k-l}{\o{t}_k-\u{t}_l}}h(s)^{\frac{(k-l)\u{t}_l}{\o{t}_k-\u{t}_l}})^{\prime}ds$ and $\delta_1=\delta_0-\int_{0}^{\theta_{0}}g_0(s)(s^{\frac{k-l}{\o{t}_k-\u{t}_l}}h(s)^{\frac{(k-l)\u{t}_l}{\o{t}_k-\u{t}_l}})^{\prime}ds$.
In \eqref{vinfty}, we let
$$Q(\theta):=\theta^{-\frac{k-l}{\o{t}_k-\u{t}_l}}\int_{\theta_{0}}^\theta C_1s^{-\beta}(s^{\frac{k-l}{\o{t}_k-\u{t}_l}}h(s)^{\frac{(k-l)\u{t}_l}{\o{t}_k-\u{t}_l}})^{\prime}ds.$$
Then if $\beta\not=\frac{k-l}{\o{t}_k-\u{t}_l}$,
\begin{align}
Q(\theta)
=&\theta^{-\frac{k-l}{\o{t}_k-\u{t}_l}}\int_{\theta_{0}}^\theta C_1s^{-\beta}(s^{\frac{k-l}{\o{t}_k-\u{t}_l}}h(s)^{\frac{(k-l)\u{t}_l}{\o{t}_k-\u{t}_l}})^{\prime}ds\nonumber\\
=&\theta^{-\frac{k-l}{\o{t}_k-\u{t}_l}}\left(C_1\theta^{\frac{k-l}{\o{t}_k-\u{t}_l}-\beta}h^{\frac{(k-l)\u{t}_l}{\o{t}_k-\u{t}_l}}(\theta)
-C_1\theta_0^{\frac{k-l}{\o{t}_k-\u{t}_l}-\beta}h^{\frac{(k-l)\u{t}_l}{\o{t}_k-\u{t}_l}}(\theta_0)+C_1\beta \int_{\theta_{0}}^\theta s^{-\beta-1}s^{\frac{k-l}{\o{t}_k-\u{t}_l}}h(s)^{\frac{(k-l)\u{t}_l}{\o{t}_k-\u{t}_l}}ds\right)\nonumber\\
=&C_2\theta^{-\beta}+C_3\theta^{-\frac{k-l}{\o{t}_k-\u{t}_l}}+C_1\beta h(\zeta_0)^{\frac{(k-l)\u{t}_l}{\o{t}_k-\u{t}_l}}\theta^{-\frac{k-l}{\o{t}_k-\u{t}_l}} \int_{\theta_{0}}^\theta s^{-\beta-1}s^{\frac{k-l}{\o{t}_k-\u{t}_l}}ds \label{Q1}\\
=&C_2\theta^{-\beta}+C_3\theta^{-\frac{k-l}{\o{t}_k-\u{t}_l}}
+\frac{C_1\beta h(\zeta_0)^{\frac{(k-l)\u{t}_l}{\o{t}_k-\u{t}_l}}}{\frac{k-l}{\o{t}_k-\u{t}_l}-\beta}\theta^{-\beta}
-\frac{C_1\beta h(\zeta_0)^{\frac{(k-l)\u{t}_l}{\o{t}_k-\u{t}_l}}}{\frac{k-l}{\o{t}_k-\u{t}_l}-\beta}\theta_0^{\frac{k-l}{\o{t}_k-\u{t}_l}-\beta}
\theta^{-\frac{k-l}{\o{t}_k-\u{t}_l}} \nonumber\\
=&C_4\theta^{-\beta}+C_5\theta^{-\frac{k-l}{\o{t}_k-\u{t}_l}},\label{Q2}
\end{align}
where $C_2:=C_2(\theta)=C_1h^{\frac{(k-l)\u{t}_l}{\o{t}_k-\u{t}_l}}(\theta)$ and $C_3=-C_1\theta_0^{\frac{k-l}{\o{t}_k-\u{t}_l}-\beta}h^{\frac{(k-l)\u{t}_l}{\o{t}_k-\u{t}_l}}(\theta_0)$. In \eqref{Q1} we employ the integration by parts and the mean value theorem of integrals and $\zeta_0\in [\theta_0,\theta]$, $C_4=C_2+\frac{C_1\beta h(\zeta_0)^{\frac{(k-l)\u{t}_l}{\o{t}_k-\u{t}_l}}}{\frac{k-l}{\o{t}_k-\u{t}_l}-\beta}$, $C_5=C_3-\frac{C_1\beta h(\zeta_0)^{\frac{(k-l)\u{t}_l}{\o{t}_k-\u{t}_l}}}{\frac{k-l}{\o{t}_k-\u{t}_l}-\beta}\theta_0^{\frac{k-l}{\o{t}_k-\u{t}_l}-\beta}.$

In \eqref{vinfty}, we set
\begin{align}\label{R}R(\theta):=&\theta^{-\frac{k-l}{\o{t}_k-\u{t}_l}}\int_{0}^\theta g_0(s)(s^{\frac{k-l}{\o{t}_k-\u{t}_l}}h(s)^{\frac{(k-l)\u{t}_l}{\o{t}_k-\u{t}_l}})^{\prime}ds
-h_0^{\frac{(k-l)\o{t}_k}{\o{t}_k-\u{t}_l}}(\theta)\nonumber\\
=&\theta^{-\frac{k-l}{\o{t}_k-\u{t}_l}}\int_{0}^\theta g_0(s)(s^{\frac{k-l}{\o{t}_k-\u{t}_l}}h(s)^{\frac{(k-l)\u{t}_l}{\o{t}_k-\u{t}_l}})^{\prime}ds
-\theta^{-\frac{k-l}{\o{t}_k-\u{t}_l}}\int_{0}^\theta g_0(s)(s^{\frac{k-l}{\o{t}_k-\u{t}_l}}h_0(s)^{\frac{(k-l)\u{t}_l}{\o{t}_k-\u{t}_l}})^{\prime}ds\nonumber\\
=&\theta^{-\frac{k-l}{\o{t}_k-\u{t}_l}}\int_{0}^\theta g_0(s)\left((s^{\frac{k-l}{\o{t}_k-\u{t}_l}}h(s)^{\frac{(k-l)\u{t}_l}{\o{t}_k-\u{t}_l}})^{\prime}
-(s^{\frac{k-l}{\o{t}_k-\u{t}_l}}h_0(s)^{\frac{(k-l)\u{t}_l}{\o{t}_k-\u{t}_l}})^{\prime}\right)ds.
\end{align}
According to \eqref{*1} and \eqref{*2}, we can have that
$$\lim_{r\to +\infty}\dfrac{(r^{\frac{k-l}{\o{t}_k-\u{t}_l}}h^{\frac{(k-l)\o{t}_k}{\o{t}_k-\u{t}_l}})^{\prime}}
{(r^{\frac{k-l}{\o{t}_k-\u{t}_l}}h_0^{\frac{(k-l)\o{t}_k}{\o{t}_k-\u{t}_l}})^{\prime}}
\dfrac{(r^{\frac{k-l}{\o{t}_k-\u{t}_l}}h_0^{\frac{(k-l)\u{t}_l}{\o{t}_k-\u{t}_l}})^{\prime}}
{(r^{\frac{k-l}{\o{t}_k-\u{t}_l}}h^{\frac{(k-l)\u{t}_l}{\o{t}_k-\u{t}_l}})^{\prime}}=\lim_{r\to+\infty}\dfrac{\o{g}(r)}{g_0(r)}=1.$$
Consequently,
\begin{equation}\label{j3}\lim_{r\to +\infty}\dfrac{(r^{\frac{k-l}{\o{t}_k-\u{t}_l}}h^{\frac{(k-l)\o{t}_k}{\o{t}_k-\u{t}_l}})^{\prime}}
{(r^{\frac{k-l}{\o{t}_k-\u{t}_l}}h_0^{\frac{(k-l)\o{t}_k}{\o{t}_k-\u{t}_l}})^{\prime}}
=\lim_{r\to +\infty}\dfrac{(r^{\frac{k-l}{\o{t}_k-\u{t}_l}}h^{\frac{(k-l)\u{t}_l}{\o{t}_k-\u{t}_l}})^{\prime}}
{(r^{\frac{k-l}{\o{t}_k-\u{t}_l}}h_0^{\frac{(k-l)\u{t}_l}{\o{t}_k-\u{t}_l}})^{\prime}}.\end{equation}
On the other hand, in light of \eqref{*1} and \eqref{*2}, we know that
$$(h_0(r))^{\frac{(k-l)\o{t}_k}{\o{t}_k-\u{t}_l}}=r^{-\frac{k-l}{\o{t}_k-\u{t}_l}}
\dint_0^r g_0(s)(s^{\frac{k-l}{\o{t}_k-\u{t}_l}}h_0(s)^{\frac{(k-l)\u{t}_l}{\o{t}_k-\u{t}_l}})^{\prime}ds,$$
and
$$(h(r))^{\frac{(k-l)\o{t}_k}{\o{t}_k-\u{t}_l}}=r^{-\frac{k-l}{\o{t}_k-\u{t}_l}}
\dint_0^r \o{g}(s)(s^{\frac{k-l}{\o{t}_k-\u{t}_l}}h(s)^{\frac{(k-l)\u{t}_l}{\o{t}_k-\u{t}_l}})^{\prime}ds.$$
As a result,
\begin{align}\label{j1}\lim_{r\to+\infty}\dfrac{(h(r))^{\frac{(k-l)\o{t}_k}{\o{t}_k-\u{t}_l}}}{(h_0(r))^{\frac{(k-l)\o{t}_k}{\o{t}_k-\u{t}_l}}}=&
\lim_{r\to+\infty}\dfrac{\o{g}(r)(r^{\frac{k-l}{\o{t}_k-\u{t}_l}}h^{\frac{(k-l)\u{t}_l}{\o{t}_k-\u{t}_l}})^{\prime}}
{g_0(r)(r^{\frac{k-l}{\o{t}_k-\u{t}_l}}h_0^{\frac{(k-l)\u{t}_l}{\o{t}_k-\u{t}_l}})^{\prime}}\nonumber\\
=&\lim_{r\to+\infty}\dfrac{(r^{\frac{k-l}{\o{t}_k-\u{t}_l}}h^{\frac{(k-l)\u{t}_l}{\o{t}_k-\u{t}_l}})^{\prime}}
{(r^{\frac{k-l}{\o{t}_k-\u{t}_l}}h_0^{\frac{(k-l)\u{t}_l}{\o{t}_k-\u{t}_l}})^{\prime}}.
\end{align}
Likewise, we also have that
\begin{align}\label{j2}\lim_{r\to+\infty}\dfrac{(h(r))^{\frac{(k-l)\u{t}_l}{\o{t}_k-\u{t}_l}}}{(h_0(r))^{\frac{(k-l)\u{t}_l}{\o{t}_k-\u{t}_l}}}=
\lim_{r\to+\infty}\dfrac{(r^{\frac{k-l}{\o{t}_k-\u{t}_l}}h^{\frac{(k-l)\o{t}_k}{\o{t}_k-\u{t}_l}})^{\prime}}
{(r^{\frac{k-l}{\o{t}_k-\u{t}_l}}h_0^{\frac{(k-l)\o{t}_k}{\o{t}_k-\u{t}_l}})^{\prime}}.
\end{align}
From \eqref{j3}, \eqref{j1} and \eqref{j2}, we get that
$$\lim_{r\to+\infty}\dfrac{(h(r))^{\frac{(k-l)\o{t}_k}{\o{t}_k-\u{t}_l}}}{(h_0(r))^{\frac{(k-l)\o{t}_k}{\o{t}_k-\u{t}_l}}}
=\lim_{r\to+\infty}\dfrac{(h(r))^{\frac{(k-l)\u{t}_l}{\o{t}_k-\u{t}_l}}}{(h_0(r))^{\frac{(k-l)\u{t}_l}{\o{t}_k-\u{t}_l}}}.$$
So
$$\lim_{r\to+\infty}\dfrac{h(r)}{h_0(r)}=1.$$
And therefore, the term $\int_{0}^\theta g_0(s)\left((s^{\frac{k-l}{\o{t}_k-\u{t}_l}}h(s)^{\frac{(k-l)\u{t}_l}{\o{t}_k-\u{t}_l}})^{\prime}
-(s^{\frac{k-l}{\o{t}_k-\u{t}_l}}h_0(s)^{\frac{(k-l)\u{t}_l}{\o{t}_k-\u{t}_l}})^{\prime}\right)ds$ in \eqref{R} is bounded  and thus
\begin{equation}\label{Q4}\theta^{-\frac{k-l}{\o{t}_k-\u{t}_l}}\int_{0}^\theta g_0(s)(s^{\frac{k-l}{\o{t}_k-\u{t}_l}}h(s)^{\frac{(k-l)\u{t}_l}{\o{t}_k-\u{t}_l}})^{\prime}ds
-h_0^{\frac{(k-l)\o{t}_k}{\o{t}_k-\u{t}_l}}(\theta)=C_{10}\theta^{-\frac{k-l}{\o{t}_k-\u{t}_l}},\end{equation}
where $c_{10}=c_{10}(\theta)=\int_{0}^\theta g_0(s)\left((s^{\frac{k-l}{\o{t}_k-\u{t}_l}}h(s)^{\frac{(k-l)\u{t}_l}{\o{t}_k-\u{t}_l}})^{\prime}
-(s^{\frac{k-l}{\o{t}_k-\u{t}_l}}h_0(s)^{\frac{(k-l)\u{t}_l}{\o{t}_k-\u{t}_l}})^{\prime}\right)ds$. Hence by \eqref{Q2} and \eqref{Q4}, we know that
\begin{align}\label{vinfty-1}
&-\int_{r}^{+\infty}\theta[h(\theta)-h_0(\theta)]d\theta\nonumber\\
=&-\int_{r}^{+\infty}\theta h_0(\theta)\left\{\left[\frac{\delta_1 \theta^{-\frac{k-l}{\o{t}_k-\u{t}_l}}}{h_0^{\frac{(k-l)\o{t}_k}{\o{t}_k-\u{t}_l}}(\theta)}+\frac{C_{10}\theta^{-\frac{k-l}{\o{t}_k-\u{t}_l}}}{h_0^{\frac{(k-l)\o{t}_k}{\o{t}_k-\u{t}_l}}(\theta)}+1
+\frac{C_4\theta^{-\beta}+C_5\theta^{-\frac{k-l}{\o{t}_k-\u{t}_l}}}{h_0^{\frac{(k-l)\o{t}_k}{\o{t}_k-\u{t}_l}}(\theta)}
\right]^{\frac{\o{t}_k-\u{t}_l}{(k-l)\o{t}_k}}
-1\right\}d\theta.
\end{align}
Thus due to the fact that $h_0$ is bounded, then \eqref{vinfty-1} becomes
\begin{align*}
&-\int_{r}^{+\infty}\theta[h(\theta)-h_0(\theta)]d\theta\nonumber\\
=&-\int_{r}^{+\infty}O(\theta^{1-\frac{k-l}{\o{t}_k-\u{t}_l}})+O(\theta^{1-\beta})d\theta\nonumber\\
=&O(r^{2-\min\{\beta,\frac{k-l}{\o{t}_k-\u{t}_l}\}}),\ \ \mbox{as}\ \ r\to+\infty.
\end{align*}

If $\beta=\frac{k-l}{\o{t}_k-\u{t}_l}$, then by \eqref{Q1},
\begin{align}
Q(\theta)
=&C_2\theta^{-\frac{k-l}{\o{t}_k-\u{t}_l}}+C_3\theta^{-\frac{k-l}{\o{t}_k-\u{t}_l}}
+C_1\frac{k-l}{\o{t}_k-\u{t}_l} h(\zeta_0)^{\frac{(k-l)\u{t}_l}{\o{t}_k-\u{t}_l}}\theta^{-\frac{k-l}{\o{t}_k-\u{t}_l}}(\ln \theta-\ln \theta_0)\nonumber\\
=&C_6\theta^{-\frac{k-l}{\o{t}_k-\u{t}_l}}\ln \theta+C_7\theta^{-\frac{k-l}{\o{t}_k-\u{t}_l}},\label{QQ3}
\end{align}
where $C_6=C_1\frac{k-l}{\o{t}_k-\u{t}_l} h(\zeta_0)^{\frac{(k-l)\u{t}_l}{\o{t}_k-\u{t}_l}},$
$C_7:=C_2+C_3-C_1\frac{k-l}{\o{t}_k-\u{t}_l} h(\zeta_0)^{\frac{(k-l)\u{t}_l}{\o{t}_k-\u{t}_l}}\ln\theta_0.$ Therefore by \eqref{vinfty}, \eqref{R} and \eqref{QQ3}, we know that
\begin{align}\label{vinfty-3}
&-\int_{r}^{+\infty}\theta[h(\theta)-h_0(\theta)]d\theta\nonumber\\
=&-\int_{r}^{+\infty}\theta h_0(\theta)\left\{\left[\frac{\delta_1 \theta^{-\frac{k-l}{\o{t}_k-\u{t}_l}}}{h_0^{\frac{(k-l)\o{t}_k}{\o{t}_k-\u{t}_l}}(\theta)}
+\frac{C_{10}\theta^{-\frac{k-l}{\o{t}_k-\u{t}_l}}}{h_0^{\frac{(k-l)\o{t}_k}{\o{t}_k-\u{t}_l}}(\theta)}+1
+\frac{C_6\theta^{-\frac{k-l}{\o{t}_k-\u{t}_l}}\ln \theta+C_7\theta^{-\frac{k-l}{\o{t}_k-\u{t}_l}}}{h_0^{\frac{(k-l)\o{t}_k}{\o{t}_k-\u{t}_l}}(\theta)}
\right]^{\frac{\o{t}_k-\u{t}_l}{(k-l)\o{t}_k}}
-1\right\}d\theta.
\end{align}
Hence by the fact that $h_0$ is bounded, then \eqref{vinfty-3} turns into
\begin{align*}
&-\int_{r}^{+\infty}\theta[h(\theta)-h_0(\theta)]d\theta\nonumber\\
=&-\int_{r}^{+\infty}O(\theta^{1-\frac{k-l}{\o{t}_k-\u{t}_l}})+O(\theta^{1-\frac{k-l}{\o{t}_k-\u{t}_l}}\ln \theta)d\theta\nonumber\\
=&O(r^{2-\frac{k-l}{\o{t}_k-\u{t}_l}}\ln r),\ \ \mbox{as}\ \ r\to+\infty.
\end{align*}

To sum up, we can get that as $r\to+\infty$,
\begin{equation}\label{vasymptotic}
w(r)=
\begin{cases}
\dint_{0}^r\theta h_0(\theta)d\theta+\mu_{\beta_1,\eta}(\delta)+O(r^{2-\min\{\beta,\frac{k-l}{\o{t}_k-\u{t}_l}\}}),\mbox{\ if\ }\beta\not=\frac{k-l}{\o{t}_k-\u{t}_l},\\
\dint_{0}^r\theta h_0(\theta)d\theta+\mu_{\beta_1,\eta}(\delta)+O(r^{2-\frac{k-l}{\o{t}_k-\u{t}_l}}\ln r),\mbox{\ if\ }\beta=\frac{k-l}{\o{t}_k-\u{t}_l}.
\end{cases}
\end{equation}

By Lemma \ref{lem2}-(ii), we know that
\begin{equation}\label{mu2}\mu_{\beta_1,\eta}(\delta)\to+\infty, \mbox{as}\ \delta\to +\infty.\end{equation}

Define
$$W(x):=W_{\beta_1,\eta,\delta,A}(x):=w(r):=w_{\beta_1,\eta,\delta}(r_{A}(x)),\ \forall \ x\in \mathbb{R}^n\b E_{\eta}.$$
Now we have the conclusion
\begin{lemma}\label{lem3}
$W$ is a smooth $k-$convex subsolution of \eqref{hq1} in $\mathbb{R}^n\b\o{E_{\eta}}$, i.e.,
$$S_{j}(D^2W(x))\geq 0,~\mbox{for any}~j=1,\dots,k,$$
and
$$\frac{S_{k}(D^2W(x))}{S_{l}(D^2W(x))}\geq g(x),~\mbox{for any}~x\in\mathbb{R}^n\b\o{E_{\eta}}.$$
\end{lemma}

\begin{proof}
According to the definition of $w(r)$, we deduce that $w'(r)=rh(r)$ and $w''(r)=h(r)+rh'(r)$. Then
\begin{align*}
\partial_{x_ix_j}W(x)=&\frac{w'(r)}{r}a_i\delta_{ij}+\frac{w''(r)-\frac{w'(r)}{r}}{r^2}(a_ix_i)(a_jx_j)\\
=&h(r)a_i\delta_{ij}+\frac{h'(r)}{r}(a_ix_i)(a_jx_j),
\end{align*}
and therefore
$$D^2W=\left(h(r)a_i\delta_{ij}+\frac{h'(r)}{r}(a_ix_i)(a_jx_j)\right)_{n\times n}.$$
Thus from Lemma \ref{A1}, we have that
\begin{align}\label{V-k-convex}
&S_j(D^2W)=\sigma_j(\lambda(D^2W))\nonumber\\
=&\sigma_j(a)h(r)^j+\frac{h'(r)}{r}h(r)^{j-1}\sum_{i=1}^{n}\sigma_{j-1;i}(a)a_i^2x_i^2\nonumber\\
=&\sigma_j(a)h(r)^j+\Lambda_j(a,x)\sigma_j(a)rh(r)^{j-1}h'(r)\nonumber\\
\geq &\sigma_j(a)h(r)^j+\o{t}_j(a)\sigma_j(a)rh(r)^{j-1}h'(r)\nonumber\\
=&\sigma_j(a)h(r)^{j-1}(h+\o{t}_j(a)rh'(r)),~j=1,\dots,k,
\end{align}
where we employ the facts that $h(r)\geq \o{g}^{\frac{1}{k-l}}(r)>0$ and $h'(r)\leq 0$ for any $r\geq 1$, due to Lemma \ref{lem2}-(i).

Since for $l<k$, $\u{t}_l\leq\u{t}_k\leq \frac{k}{n}\leq \o{t}_k$, then
$$-\frac{\u{t}_l}{\o{t}_k}\geq -1.$$
Thus $$0\leq \frac{h^{k-l}-\o{g}(r)}{h^{k-l}-\o{g}(r)\frac{\u{t}_l}{\o{t}_k}}\leq 1\leq \frac{\o{t}_k}{\o{t}_j}, j\leq k.$$
And so
\begin{align*}h^{\prime}=&-\frac{1}{r}\dfrac{h}{\o{t}_k}\dfrac{h(r)^{k-l}-\o{g}(r)}{h(r)^{k-l}-\o{g}(r)\frac{\u{t}_l}{\o{t}_k}}\\
\geq&-\frac{1}{r}\dfrac{h}{\o{t}_k}\frac{\o{t}_k}{\o{t}_j}\\
=&-\frac{1}{r}\dfrac{h}{\o{t}_j}.
\end{align*}
Therefore,
$$h+\o{t}_jrh^{\prime}\geq 0.$$
Hence by \eqref{V-k-convex},
$$S_j(D^2w)\geq 0,~j=1,\dots,k.$$

On the other hand, by \eqref{w1} and the fact that $\sigma_{k}(a)=\sigma_{l}(a)$ for $A\in \mathcal{A}_{k,l}$, we obtain that
\begin{align*}
&\frac{S_k(D^2W(x))}{S_l(D^2W(x))}=\frac{\sigma_k(\lambda(D^2W(x)))}{\sigma_l(\lambda(D^2W(x)))}\\
=&\frac{\sigma_k(a)h(r)^k+\frac{h'(r)}{r}h(r)^{k-1}\sum_{i=1}^{n}\sigma_{k-1;i}(a)a_i^2x_i^2}{\sigma_l(a)h(r)^l+\frac{h'(r)}{r}h(r)^{l-1}\sum_{i=1}^{n}\sigma_{l-1;i}(a)a_i^2x_i^2}\\
=&\frac{\sigma_k(a)h(r)^k+\Lambda_k(a,x)\sigma_k(a)rh(r)^{k-1}h'}{\sigma_l(a)h(r)^l+\Lambda_l(a,x)\sigma_l(a)rh(r)^{l-1}h'}\\
\geq &\frac{\sigma_k(a)h(r)^k+\o{t}_k(a)\sigma_k(a)rh(r)^{k-1}h'}{\sigma_l(a)h(r)^l+\u{t}_l(a)\sigma_l(a)rh(r)^{l-1}h'}\\
=&\frac{h(r)^k+\o{t}_k(a)rh(r)^{k-1}h'}{h(r)^l+\u{t}_l(a)rh(r)^{l-1}h'}\\
=&\o{g}(r)\geq g(x),~x\in\mathbb{R}^n\b\o{E_{\eta}}.
\end{align*}
Then we complete the proof.
\end{proof}

In light of Lemma \ref{lem5}, we know that $$\frac{\u{t}_l}{\o{t}_k}< 1,l<k.$$
Let
$$\left(\frac{\u{t}_l}{\o{t}_k}\u{g}(1)\right)^{\frac{1}{k-l}}<\tau<(\u{g}(1))^{\frac{1}{k-l}}.$$
\begin{lemma}\label{lem4}
Let $0\leq l<k\leq n, n\geq 3,A\in \mathcal{A}_{k,l}, a:=(a_1,a_2,\dots,a_n):=\lambda(A),0<a_1\leq a_2\leq \dots\leq a_n$. Then the problem
\begin{equation}\label{q}
\begin{cases}
\dfrac{H(r)^k+\o{t}_k r H(r)^{k-1}H'(r)}{H(r)^l+\u{t}_l r H(r)^{l-1}H'(r)}=\u{g}(r),\ r>1,\\
H(1)=\tau,
\end{cases}
\end{equation}
has a smooth solution $H(r)=H(r,\tau)$ on $[1,+\infty)$ satisfying

{\rm(i)}~ $\frac{\u{t}_l}{\o{t}_k}\u{g}(r)<H^{k-l}(r,\tau)<\u{g}(r),\partial_{r}H(r,\tau)\geq 0$ for $r\geq 1$.

{\rm(ii)}~$H(r,\tau)$ is continuous and strictly increasing with respect to $\tau$.

\end{lemma}

\begin{proof}
For brevity, we sometimes write $H(r)$ or $H(r,\tau)$ when there is no confusion. From \eqref{q}, we have
\begin{equation}\label{q1}\begin{cases}\displaystyle\frac{\mbox{d} H}{\mbox{d}r}=-\frac{1}{r}\dfrac{H}{\o{t}_k}\dfrac{H(r)^{k-l}-\u{g}(r)}{H(r)^{k-l}-\u{g}(r)\frac{\u{t}_l}{\o{t}_k}},\ r>1,\\
H(1)=\tau.
\end{cases}
\end{equation}
Since $\frac{\u{t}_l}{\o{t}_k}\u{g}(1)<\tau^{k-l}<\u{g}(1)$ and $\u{g}(r)$ is strictly increasing, by the existence, uniqueness and extension theorem for the solution of the initial value problem of the ODE, we can get that the problem has a smooth solution $H(r,\delta)$ satisfying $\frac{\u{t}_l}{\o{t}_k}\u{g}(r)<H^{k-l}(r,\tau)<\u{g}(r),$ and $\partial_{r}H(r,\tau)\geq 0$. So assertion (i) of the lemma is proved.
 Let
 $$p(H):=p(H(r,\tau)):=\dfrac{H(r,\tau)}{\o{t}_k}\dfrac{H(r,\tau)^{k-l}-\u{g}(r)}{H(r,\tau)^{k-l}-\u{g}(r)\frac{\u{t}_l}{\o{t}_k}}.$$
 Then \eqref{q1} is
 \begin{equation}\label{H1}\begin{cases}\displaystyle\frac{\partial H}{\partial r}=-\frac{1}{r}p(H(r,\tau)),\ r>1,\\
H(1,\tau)=\tau.
\end{cases}
\end{equation}
Differentiating \eqref{H1} for $\tau$, we get that
$$\begin{cases}\displaystyle\frac{\partial^2 H}{\partial r\partial \tau}=-\frac{1}{r}p^{\prime}(H)\dfrac{\partial H}{\partial \tau},\vspace{2mm}\\
\dfrac{\partial H(1,\tau)}{\partial \tau}=1.
\end{cases}
$$
Let $$q(r):=\dfrac{\partial H(r,\tau)}{\partial \tau}.$$
Then
$$\begin{cases}\displaystyle\frac{\mbox{d}q}{\mbox{d} r}=-\frac{1}{r}p^{\prime}(H)q,\\
q(1)=1.
\end{cases}
$$
So
$$\dfrac{\mbox{d}q}{q}=-\frac{1}{r}p^{\prime}(H)\mbox{d} r,$$
and thus
$$\dfrac{\partial H(r,\tau)}{\partial \tau}=q(r)=\mbox{exp}\dint_1^r(-\frac{1}{s})p^{\prime}(H(s,\tau))ds>0.$$
Therefore $H(r,\tau)$ is strictly increasing in $\tau.$ The lemma is proved.
\end{proof}

\begin{remark}\label{remsuper1}
By the extension theorem of solutions, we can know that the solution $H$ in \eqref{q} can be extended to the left and is well defined in $[0,+\infty).$
\end{remark}

\begin{remark}
If $g(r)\equiv 1$, then we choose $\o{g}(r)=\u{g}(r)=1,\delta=\tau=1$ and so \eqref{w} and \eqref{q1} all have a solution $h(r)=H(r)=1$.
\end{remark}

From \eqref{q},
$$\dfrac{\frac{k-l}{\o{t}_k-\u{t}_l}r^{\frac{k-l}{\o{t}_k-\u{t}_l}-1}H(r)^{\frac{k\u{t}_l-l\o{t}_k}{\o{t}_k-\u{t}_l}}[H(r)^k+\o{t}_k r H(r)^{k-1}H'(r)]}{\frac{k-l}{\o{t}_k-\u{t}_l}r^{\frac{k-l}{\o{t}_k-\u{t}_l}-1}H(r)^{\frac{k\u{t}_l-l\o{t}_k}{\o{t}_k-\u{t}_l}}[H(r)^l+\u{t}_l r H(r)^{l-1}H'(r)]}=\u{g}(r),\ r>1.$$
So
$$\dfrac{(r^{\frac{k-l}{\o{t}_k-\u{t}_l}}H^{\frac{(k-l)\o{t}_k}{\o{t}_k-\u{t}_l}})^{\prime}}
{(r^{\frac{k-l}{\o{t}_k-\u{t}_l}}H^{\frac{(k-l)\u{t}_l}{\o{t}_k-\u{t}_l}})^{\prime}}=\u{g}(r),r>1,$$
that is
$$(r^{\frac{k-l}{\o{t}_k-\u{t}_l}}H^{\frac{(k-l)\o{t}_k}{\o{t}_k-\u{t}_l}})^{\prime}
=\u{g}(r)(r^{\frac{k-l}{\o{t}_k-\u{t}_l}}H^{\frac{(k-l)\u{t}_l}{\o{t}_k-\u{t}_l}})^{\prime},r>1.$$
Integrating the above equality from $1$ to $r$, we have that
$$r^{\frac{k-l}{\o{t}_k-\u{t}_l}}H^{\frac{(k-l)\o{t}_k}{\o{t}_k-\u{t}_l}}-\tau^{\frac{(k-l)\o{t}_k}{\o{t}_k-\u{t}_l}}
=\dint_1^r \u{g}(s)(s^{\frac{k-l}{\o{t}_k-\u{t}_l}}H(s)^{\frac{(k-l)\u{t}_l}{\o{t}_k-\u{t}_l}})^{\prime}ds.$$
As a result,
\begin{equation}\label{H(r)}H(r)=\left(\tau^{\frac{(k-l)\o{t}_k}{\o{t}_k-\u{t}_l}}r^{-\frac{k-l}{\o{t}_k-\u{t}_l}}
+r^{-\frac{k-l}{\o{t}_k-\u{t}_l}}
\dint_1^r\u{g}(s)(s^{\frac{k-l}{\o{t}_k-\u{t}_l}}H(s)^{\frac{(k-l)\u{t}_l}{\o{t}_k-\u{t}_l}})^{\prime}ds\right)^{\frac{\o{t}_k-\u{t}_l}{(k-l)\o{t}_k}}.
\end{equation}

Define for any constant $\beta_2$,
$$\psi(r):=\psi_{\beta_2,\eta,\tau}(r):=\beta_2+\int_{\eta}^{r}\theta H(\theta,\tau)d\theta,\ \forall \ r\geq \eta\geq 1, $$
and
$$\Psi(x):=\Psi_{\beta_2,\eta,\tau,A}(x):=\psi(r):=\psi_{\beta_2,\eta,\tau}(r_{A}(x)),\ \forall \ x\in \mathbb{R}^n\b E_{\eta}.$$
Similar to \eqref{vasymptotic}, as $r\to+\infty,$
\begin{equation}\label{lambdaasymptotic}
\psi(r)=
\begin{cases}
\dint_{0}^r\theta h_0(\theta)d\theta+\nu_{\beta_2,\eta}(\tau)+O(r^{2-\min\{\beta,\frac{k-l}{\o{t}_k-\u{t}_l}\}}),\mbox{\ if\ }\beta\not=\frac{k-l}{\o{t}_k-\u{t}_l},\\
\dint_{0}^r\theta h_0(\theta)d\theta+\nu_{\beta_2,\eta}(\tau)+O(r^{2-\frac{k-l}{\o{t}_k-\u{t}_l}}\ln r),\mbox{\ if\ }\beta=\frac{k-l}{\o{t}_k-\u{t}_l},
\end{cases}
\end{equation}
where
$$
h_0(\theta)=\left(\theta^{-\frac{k-l}{\o{t}_k-\u{t}_l}}
\dint_0^\theta g_0(s)(s^{\frac{k-l}{\o{t}_k-\u{t}_l}}h_0(s)^{\frac{(k-l)\u{t}_l}{\o{t}_k-\u{t}_l}})^{\prime}ds\right)^{\frac{\o{t}_k-\u{t}_l}{(k-l)\o{t}_k}},
$$
and
$$\nu_{\beta_2,\eta}(\tau):=\beta_2-\int_{0}^\eta \theta h_0(\theta)d\theta+\int_{\eta}^{+\infty}\theta[H(\theta)-h_0(\theta)]d\theta.
$$
From Lemma \ref{A1}, we get that
\begin{align*}
S_j(D^2\Psi)=&\sigma_j(\lambda(D^2\Psi))\\
=&\sigma_j(a)H(r)^j+\frac{H'(r)}{r}H(r)^{j-1}\sum_{i=1}^{n}\sigma_{j-1;i}(a)a_i^2x_i^2\\
=&\sigma_j(a)H(r)^j+\Lambda_j(a,x)\sigma_j(a)rH(r)^{j-1}H'(r)\\
\geq& 0,~1\leq j\leq k,
\end{align*}
where we use the facts that $H'\geq 0$ by Lemma \ref{lem4}-(i).
Furthermore, from \eqref{q}, we get that for $A\in \mathcal{A}_{k,l}$,
\begin{align}\label{sup}
\frac{S_k(D^2\Psi)}{S_l(D^2\Psi)}=&\frac{\sigma_k(\lambda(D^2\Psi))}{\sigma_l(\lambda(D^2\Psi))}\nonumber\\
=&\frac{\sigma_k(a)H(r)^k+\frac{H'(r)}{r}H(r)^{k-1}\sum_{i=1}^{n}\sigma_{k-1;i}(a)a_i^2x_i^2}
{\sigma_l(a)H(r)^l+\frac{H'(r)}{r}H(r)^{l-1}\sum_{i=1}^{n}\sigma_{l-1;i}(a)a_i^2x_i^2}\nonumber\\
=&\frac{\sigma_k(a)H(r)^k+\Lambda_k(a,x)\sigma_k(a)rH(r)^{k-1}H'}{\sigma_l(a)H(r)^l+\Lambda_l(a,x)\sigma_l(a)rH(r)^{l-1}H'}\nonumber\\
=&\frac{H(r)^k+\Lambda_k(a,x)rH(r)^{k-1}H'}{H(r)^l+\Lambda_l(a,x)rH(r)^{l-1}H'}\nonumber\\
\leq &\frac{H(r)^k+\o{t}_k rH(r)^{k-1}H'}{H(r)^l+\o{t}_l rH(r)^{l-1}H'}\nonumber\\
=&\u{g}(r)\leq g(x),\ \forall \ x\in \mathbb{R}^n\b E_{\eta},
\end{align}
in which we apply the facts that $\sigma_k(a)=\sigma_l(a).$ So we have that

\begin{theorem}\label{thm2}
$\Psi$ is a $k-$convex supersolution of \eqref{hq1} in $\mathbb{R}^n\b E_{\eta}$.
\end{theorem}

\section{Proof of Theorem \ref{thm1}}

Suppose that $E_{1}\subset\subset \Omega\subset\subset E_{r_0}\subset\subset E_{R_0}$.  For $0\leq l<k\leq n, n\geq3$, $A\in \mathcal{A}_{k,l}$, we first give the following lemma.
\begin{lemma}\label{lem1}
Suppose that $\phi\in C^{2}(\partial\Omega)$. Then there exists some constant $C$, depending only on $g,~ n, ||\phi||_{C^{2}(\partial\Omega)},$
the upper bound of $A$, the diameter and the convexity of $\Omega$, and the $C^{2}$ norm of $\partial\Omega$, such that, for each
$\xi\in\partial\Omega$, there exists $\overline{x}(\xi)\in\mathbb{R}^{n}$ such that $|\overline{x}(\xi)|\leq C$,
$$\rho_{\xi}<\phi~~~~on~~~\partial\Omega\backslash\{\xi\}~~~~~\mbox{and}~~~~\rho_{\xi}(\xi)=\phi(\xi),$$
where
$$
\rho_{\xi}(x)=\phi(\xi)+\frac{\Xi}{2}[(x-\overline{x}(\xi))^{T}A(x-\overline{x}(\xi))
-(\xi-\overline{x}(\xi))^{T}A(\xi-\overline{x}(\xi))],~~~x\in\mathbb{R}^{n},
$$
and $\frac{C_n^k\Xi^k}{C_n^l\Xi^l}>\sup\limits_{E_{R_0}}\overline{g}$.
\end{lemma}
\begin{proof}
See \cite{BLL}.
\end{proof}

\begin{proof}[Proof of Theorem \ref{thm1}]
Without loss of generality, we may assume that $A=\mbox{diag}(a_{1},\cdots,a_{n})\in \mathcal{A}_{k,l}$, $0<a_1\leq a_2\leq \cdots\leq a_n$ and $b=0$. Let
$$\zeta_1:=\min_{\substack{x\in \o{E_{r_0}}\b\Omega \\ \xi\in \partial \Omega}}\rho_{\xi}(x),$$
$$\varphi(x):=\max_{\xi\in \partial\Omega}\rho_{\xi}(x),$$
and for $r_{A}(x)\geq 1,\delta> \sup_{r\in[1,+\infty)} \o{g}^{\frac{1}{k-l}}(r),$
$$W_{\delta}(x):=\zeta_1+\int_{r_0}^{r_A(x)}\theta h(\theta,\delta)d\theta,\ x\in\mathbb{R}^n\b\{0\},$$
where $\rho_{\xi}(x)$ and $h(r,\delta)$ are respectively given by Lemma \ref{lem1} and Lemma \ref{lem2}.

Since $\rho_{\xi}(x)$ satisfies
$$\frac{S_k(D^2\rho_{\xi}(x))}{S_l(D^2\rho_{\xi}(x))}\geq g(x),~~x\in E_{R_0},$$
then $\varphi(x)$ satisfies in the viscosity sense that
\begin{equation}\label{6}\frac{S_k(D^2\varphi(x))}{S_l(D^2\varphi(x))}\geq g(x),~~x\in E_{R_0}.\end{equation}
Moreover,
\begin{equation}\label{varphiboun}\varphi=\phi\ \mbox{on}\ \partial \Omega.\end{equation}
By Lemma \ref{lem3}, we know that $W_{\delta}$ is a smooth $k-$convex subsolution of \eqref{hq1}, that is
\begin{equation}\label{7}\frac{S_{k}(D^2W_{\delta}(x))}{S_{l}(D^2W_{\delta}(x))}\geq g(x)\ \mbox{in}\ \mathbb{R}^n\b\o{\Omega}.\end{equation}
Since $\Omega\subset\subset E_{r_0}$, then we conclude that
\begin{equation}\label{5}W_{\delta}\leq \zeta_1\leq \rho_{\xi}\leq \varphi\ \ \mbox{on}\ \ \o{E_{r_0}}\b\Omega.\end{equation}

In addition, by Lemma \ref{lem2}, we deduce that $W_{\delta}(x)$ is strictly increasing in $\delta$ and
$$\lim_{\delta\to+\infty}W_{\delta}(x)=+\infty \ \ \mbox{for any}\ \ r_{A}(x)>r_0.$$
In light of \eqref{vasymptotic}, we have that as $|x|\to+\infty,$
$$
W_{\delta}(x)=\begin{cases}
\dint_{0}^{r_A(x)}\theta h_0(\theta)d\theta+\mu(\delta)+O(|x|^{2-\min\{\beta,\frac{k-l}{\o{t}_k-\u{t}_l}\}}),\mbox{\ if\ }\beta\not=\frac{k-l}{\o{t}_k-\u{t}_l},\\
\dint_{0}^{r_A(x)}\theta h_0(\theta)d\theta+\mu(\delta)+O(|x|^{2-\frac{k-l}{\o{t}_k-\u{t}_l}}\ln |x|),\mbox{\ if\ }\beta=\frac{k-l}{\o{t}_k-\u{t}_l}.
\end{cases}
$$
where
\begin{align*}&\mu(\delta):=\zeta_1-\int_{0}^{r_0} \theta h_0(\theta)d\theta+\int_{r_0}^{+\infty}\theta[h(\theta)-h_0(\theta)]d\theta,
\end{align*}
and we use the fact that $x^{T}Ax=O(|x|^2)$ as $|x|\to +\infty$.

Define
$$\o{u}_{\zeta_2,\tau}(x):=\zeta_2+\int_{1}^{r_A(x)}\theta H(\theta,\tau)d\theta,\ \forall \ x\in \mathbb{R}^n\b \Omega,$$
where $\zeta_2$ is any constant and $\frac{\u{t}_l}{\o{t}_k}\u{g}(1)<\tau^{k-l}<\u{g}(1)$.
Then by \eqref{sup}, we know that
\begin{equation}\label{super}\frac{S_k(D^2\o{u}_{\zeta_2,\tau})}{S_l(D^2\o{u}_{\zeta_2,\tau})}\leq g(x),\ \forall \ x\in \mathbb{R}^n\b \Omega,\end{equation}
and by \eqref{lambdaasymptotic}, as $|x|\to+\infty,$
\begin{equation*}
\o{u}_{\zeta_2,\tau}(x)=
\begin{cases}
\dint_{0}^{r_A(x)}\theta h_0(\theta)d\theta
+\nu_{\zeta_2}(\tau)+O(|x|^{2-\min\{\beta,\frac{k-l}{\o{t}_k-\u{t}_l}\}}),\mbox{\ if\ }\beta\not=\frac{k-l}{\o{t}_k-\u{t}_l},\\
\dint_{0}^{r_A(x)}\theta h_0(\theta)d\theta
+\nu_{\zeta_2}(\tau)+O(|x|^{2-\frac{k-l}{\o{t}_k-\u{t}_l}}\ln |x|),\mbox{\ if\ }\beta=\frac{k-l}{\o{t}_k-\u{t}_l},
\end{cases}
\end{equation*}
where
\begin{align*}&\nu_{\zeta_2}(\tau):=\zeta_2
-\int_{0}^1\theta h_0(\theta)d\theta+\int_{1}^{+\infty}\theta[H(\theta,\tau)-h_0(\theta)]d\theta
\end{align*}
is convergent.

Since $W_{\delta}$ is strictly increasing in $\delta$, then for $R_0>r_0$, there exists some $\hat{\delta}>\sup_{r\in[1,+\infty)} \o{g}^{\frac{1}{k-l}}(r)$ such that
$$\min_{\partial E_{R_0}}W_{\hat{\delta}}>\max_{\partial E_{R_0}}\varphi.$$
Therefore we have that
\begin{equation}\label{ww1}W_{\hat{\delta}}>\varphi\ \ \mbox{on}\ \ \partial E_{R_0}.\end{equation}
Clearly, $\mu(\delta)$ is strictly increasing in $\delta$. By \eqref{mu2}, we can know that
$$\lim_{\delta\to+\infty}\mu(\delta)=+\infty.$$
Let
$$\hat{c}=\hat{c}(\tau)=\sup_{E_{R_0\b \Omega}}\varphi-\int_{0}^1\theta h_0(\theta)d\theta+\int_{1}^{+\infty}\theta[H(\theta)-h_0(\theta)]d\theta,$$
and
$$\tilde{c}:=\max\{\hat{c},\mu(\hat{\delta}),\max_{\substack{\xi\in \partial\Omega\\x\in \o{E_{R_0}}\b\Omega}}\rho_{\xi}(x)\}.$$
So for any
\begin{equation}\label{c} c>\tilde{c},\end{equation}
there is a unique $\delta(c)$ such that $\mu(\delta(c))=c.$ Consequently, we get that as $|x|\to+\infty,$
\begin{equation}\label{Vinfty}
W_{\delta(c)}(x)=
\begin{cases}
\dint_{0}^{r_A(x)}\theta h_0(\theta)d\theta
+c+O(|x|^{2-\min\{\beta,\frac{k-l}{\o{t}_k-\u{t}_l}\}}),\mbox{\ if\ }\beta\not=\frac{k-l}{\o{t}_k-\u{t}_l},\\
\dint_{0}^{r_A(x)}\theta h_0(\theta)d\theta
+c+O(|x|^{2-\frac{k-l}{\o{t}_k-\u{t}_l}}\ln |x|),\mbox{\ if\ }\beta=\frac{k-l}{\o{t}_k-\u{t}_l},
\end{cases}
\end{equation}
and
$$\delta(c)=\mu^{-1}(c)> \mu^{-1}(\mu(\hat{\delta}))=\hat{\delta}.$$
Due to the monotonicity of $W_{\delta}$ in $\delta$ and \eqref{ww1}, we conclude that
\begin{equation}\label{8}W_{\delta(c)}\geq W_{\hat{\delta}}>\varphi\ \ \mbox{on}\ \ \partial E_{R_0}.\end{equation}

For the constant $c$ in \eqref{c}, set $\nu_{\zeta_2}(\tau)=c$, then as $|x|\to+\infty,$
\begin{equation}\label{supasymptotic}
\o{u}_{\zeta_2,\tau}(x)=\begin{cases}
\dint_{0}^{r_A(x)}\theta h_0(\theta)d\theta
+c+O(|x|^{2-\min\{\beta,\frac{k-l}{\o{t}_k-\u{t}_l}\}}),\mbox{\ if\ }\beta\not=\frac{k-l}{\o{t}_k-\u{t}_l},\\
\dint_{0}^{r_A(x)}\theta h_0(\theta)d\theta
+c+O(|x|^{2-\frac{k-l}{\o{t}_k-\u{t}_l}}\ln |x|),\mbox{\ if\ }\beta=\frac{k-l}{\o{t}_k-\u{t}_l},
\end{cases}
\end{equation}
and
\begin{align}\label{zeta2}\zeta_2:=&\zeta_2(\tau)\nonumber\\
:=&c+\int_{0}^1\theta h_0(\theta)d\theta-\int_{1}^{+\infty}\theta[H(\theta,\tau)-h_0(\theta)]d\theta.
\end{align}

Define
$$
\u{u}(x):=
\begin{cases}
\max\{W_{\delta(c)}(x),\varphi(x)\},&x\in E_{R_0}\b\Omega,\\
W_{\delta(c)}(x),&x\in \mathbb{R}^n\b E_{R_0}.
\end{cases}$$
Then by \eqref{8}, we can get that $\u{u}\in C^0(\mathbb{R}^n\b\Omega)$ and
by \eqref{6}, \eqref{7} and Lemma \ref{vis sub}, we know that
$$\dfrac{S_{k}(D^2\u{u}(x))}{S_{l}(D^2\u{u}(x))}\geq g(x)\ \ \mbox{in}\ \ \mathbb{R}^n\b\o{\Omega}.$$
In light of \eqref{5} and \eqref{varphiboun}, we obtain that
$$\u{u}=\phi\ \ \mbox{on} \ \ \partial \Omega.$$
In addition, from \eqref{Vinfty}, we have that as $|x|\to+\infty$, $\u{u}(x)$ satisfies the asymptotic behavior \eqref{supasymptotic}.

By the definition of $\tilde{c}$, $\o{u}_{\zeta_2}(x)$ and $\varphi(x)$, we have that
\begin{equation}\label{supvarphi}\o{u}_{\zeta_2,\tau}(x)\geq \zeta_2\geq \varphi(x),\ \ x\in E_{R_0}\b\Omega.\end{equation}
Again by \eqref{5},
$$W_{\delta(c)}(x)\leq \varphi(x)\leq \o{u}_{\zeta_2,\tau}(x),\ \ x\in \partial\Omega.$$
Due to \eqref{7}, \eqref{super}, \eqref{Vinfty}, \eqref{supasymptotic} and the comparison principle, we conclude that
\begin{equation}\label{supV}W_{\delta(c)}(x)\leq \o{u}_{\zeta_2,\tau}(x),\ \ x\in\mathbb{R}^n\b\Omega.\end{equation}
Denote
$$\o{u}(x):=\o{u}_{\zeta_2,\tau}(x),\ \ x\in \mathbb{R}^n\b\Omega.$$
In light of \eqref{supvarphi}, \eqref{supV} and the definition of $\u{u}$, we know that
$$\u{u}(x)\leq \o{u}(x),\ \ x\in \mathbb{R}^n\b\Omega.$$

For any $c>\tilde{c}$, let $\mathcal{S}_c$ denote the set of $\varrho\in C^0(\mathbb{R}^n\b\Omega)$ which are viscosity subsolutions of \eqref{hq1} and \eqref{hq2} satisfying $\varrho=\phi$ on $\partial \Omega$ and $\varrho\leq \o{u}$ in $\mathbb{R}^n\b\Omega$. Apparently, $\u{u}\in \mathcal{S}_c$ and so $\mathcal{S}_c\neq\emptyset.$ Define
$$u(x):=\sup\{\varrho(x)|\varrho\in \mathcal{S}_c\},\ \ x\in\mathbb{R}^n\b\Omega.$$
Then
$$\u{u}\leq u\leq \o{u}\ \ \mbox{in}\ \ \mathbb{R}^n\b\Omega.$$
Hence by the asymptotic behavior of $\u{u}$ and $\o{u}$, we have that as $|x|\to+\infty,$
\begin{equation*}\label{supasymptotic1}
u(x)=\begin{cases}
\dint_{0}^{r_A(x)}\theta h_0(\theta)d\theta
+c+O(|x|^{2-\min\{\beta,\frac{k-l}{\o{t}_k-\u{t}_l}\}}),\mbox{\ if\ }\beta\not=\frac{k-l}{\o{t}_k-\u{t}_l},\\
\dint_{0}^{r_A(x)}\theta h_0(\theta)
+c+O(|x|^{2-\frac{k-l}{\o{t}_k-\u{t}_l}}\ln |x|),\mbox{\ if\ }\beta=\frac{k-l}{\o{t}_k-\u{t}_l}.
\end{cases}
\end{equation*}

Next, we will show that $u=\phi$ on $\partial\Omega.$ On the one hand, since $\u{u}=\phi$ on $\partial\Omega,$ then
$$\liminf_{x\to\xi}u(x)\geq \lim_{x\to\xi}\u{u}(x)=\phi(\xi),\ \ \xi\in \partial\Omega.$$
On the other hand, we will prove that
$$\limsup_{x\to\xi}u(x)\leq \phi(\xi),\ \ \xi\in \partial\Omega.$$
Let $\vartheta\in C^2(\o{E_{R_0}\b\Omega})$ satisfy
$$\begin{cases}
\Delta \vartheta=0,&\ \mbox{in}\ E_{R_0}\b\o{\Omega},\\
\vartheta=\phi,&\ \mbox{on}\ \partial\Omega,\\
\vartheta=\max_{\partial E_{R_0}}\o{u},&\ \mbox{on}\ \partial E_{R_0}.
\end{cases}$$
Due to Newtonian inequalities, for any $\varrho\in \mathcal{S}_c$, we have $\Delta \varrho\geq 0$ in the viscosity sense. Moreover, $\varrho\leq \vartheta$ on $\partial (E_{R_0}\b\Omega)$. Then by the comparison principle, we deduce that
$$\varrho\leq \vartheta\ \ \mbox{in}\ \ E_{R_0}\b\Omega.$$
So
$$u\leq \vartheta\ \ \mbox{in}\ \ E_{R_0}\b\Omega.$$
Therefore
$$\limsup_{x\to\xi}u(x)\leq \lim_{x\to\xi}\vartheta(x)=\phi(\xi)\ \ \mbox{for}\ \ \xi\in\partial\Omega.$$

In the end, we prove that $u\in C^0(\mathbb{R}^n\b\Omega)$ is a viscosity solution to \eqref{hq1}. For any $x\in\mathbb{R}^n\backslash\o{\Omega},$ choose some
$\varepsilon>0$ such that
$B_{\varepsilon}=B_{\varepsilon}(x)\subset \mathbb{R}^n\backslash\o{\Omega}.$ From
Lemma \ref{existence}, the Dirichlet problem
\begin{equation}\label{t9}
\begin{cases}
\dfrac{S_k(D^2\tilde{u}(y))}{S_l(D^2\tilde{u}(y))}=g(y),&\ \ \ y\in B_{\varepsilon},\\
 \tilde{u}=u,&\ \ y\in \partial B_{\varepsilon}
\end{cases}
\end{equation}has a unique $k-$convex viscosity solution $\tilde{u}\in
C^0(\o{B_{\varepsilon}}).$ By the comparison principle,
 $u\leq \tilde {u}$ in $B_{\varepsilon}$. Define
$$\tilde{w}(y)=
\begin{cases}
\tilde{u}(y),&\ \ y\in B_{\varepsilon},\\
 u(y),&\ \ y\in (\mathbb{R}^n\backslash\Omega) \b B_{\varepsilon},
\end{cases}
$$ Then $\tilde{w}\in \mathcal{S}_{c}.$ Indeed, from the comparison principle, $\tilde{u}(y)\leq \o{u}(y)$ in $\o{B_{\varepsilon}}$ and so $\tilde{w}(y)\leq \o{u}(y)$ in $\mathbb{R}^n\b B_{\varepsilon}.$
By Lemma \ref{vis sub}, we have $\frac{S_k(D^2\tilde{w}(y))}{S_l(D^2\tilde{w}(y))}\geq g(y)$ in $\mathbb{R}^n\backslash\o{\Omega}.$ Therefore $\tilde{w}\in \mathcal{S}_{c}.$ And thus by the definition of $u$, $u(y)\geq \tilde{w}(y)$ in $\mathbb{R}^n\backslash \Omega$ and so $u(y)\geq \tilde{u}(y)$ in $B_{\varepsilon}.$ Hence
$$
u(y)\equiv \tilde{u}(y),~\forall~y\in B_{\varepsilon}.
$$
But $\tilde{u}$ satisfies (\ref{t9}), then we have in the viscosity sense,
$$\dfrac{S_k(D^2u(y))}{S_l(D^2u(y))}=g(y),~\forall~y\in B_{\varepsilon}.$$
In particular, we have
$$\dfrac{S_k(D^2u(x))}{S_l(D^2u(x))}=g(x).$$
Because $x$ is arbitrary, we know that $u$ is a viscosity solution of \eqref{hq1}. Theorem \ref{thm1} is proved.
\end{proof}

\section{Appendix}

\begin{lemma}\label{Uex} Let $2\leq k\leq m\leq n, n\geq 3, 0\leq l<k$, $$\alpha>\left(\dfrac{lC_n^l}{kC_n^k}\right)^{\frac{l}{k-l}}\left(\frac{l}{k}-1\right)C_n^l$$ and
\begin{equation}\label{U1}
\left(\dfrac{lC_n^l}{kC_n^k}\right)^{\frac{1}{k-l}}<U<\gamma_m,
\end{equation}
where $\gamma_m$ is the same as \eqref{gammam}. Then
\begin{equation}\label{Uex-equation}U^k-\frac{C_n^l}{C_n^k}U^l-\frac{\alpha}{C_n^kr^n}=0,\ r>1.\end{equation}
has a unique solution $U=U(r,\alpha)$.
\end{lemma}

\begin{proof} Let $$F(U)=U^k-\frac{C_n^l}{C_n^k}U^l-\frac{\alpha}{C_n^kr^n},\ r>1.$$
Then $$F^{\prime}(U)=kU^{k-1}-l\frac{C_n^l}{C_n^k}U^{l-1}>0.$$
So $F(u)$ is strictly increasing in $U$. Moreover, $$F\left(\left(\dfrac{lC_n^l}{kC_n^k}\right)^{\frac{1}{k-l}}\right)<0$$ and $F(\gamma_m)>0.$
Hence \eqref{Uex-equation} has a unique solution.
\end{proof}

\begin{lemma}\label{A1}
If $M=(p_i\delta_{ij}+sq_iq_j)_{n\times n}$ with $p,q\in \mathbb{R}^n$ and $s\in \mathbb{R}$, then
$$\sigma_k(\lambda(M))=\sigma_k(p)+s\sum_{i=1}^n\sigma_{k-1;i}(p)q_i^2,~k=1,\dots,n.$$
\end{lemma}
\begin{proof}
See \cite{BLL}.
\end{proof}

\begin{lemma}\label{existence}
Let $B$ be a ball in $\mathbb{R}^n$ and $f\in C^0(\o{B})$ be
nonnegative. Suppose $\u{u}\in C^0(\o{B})$ satisfies in the
viscosity sense
 $S_k(D^2\u{u})\geq f(x)$ $\mbox{\ in\ } B.$ Then the Dirichlet problem
\begin{align*}
 \dfrac{S_{k}(D^{2}u)}{S_{l}(D^{2}u)}=&f(x),\ \ x\in B,\\
u=&\u{u}(x),\ \ x\in \partial B
\end{align*}
has a
unique $k-$convex viscosity solution $u\in C^0(\o{B})$.
\end{lemma}

\begin{proof}
See \cite{D}.
\end{proof}

\begin{lemma}\label{vis sub}
Let $D$ be a domain in $\mathbb{R}^n$ and $f\in
C^{0}(\mathbb{R}^n)$ be nonnegative. Assume $k-$ convex functions
$v\in C^{0}(\o{D}),u\in C^{0}(\mathbb{R}^n)$ satisfy
respectively
$$ \dfrac{S_{k}(D^{2}v)}{S_{l}(D^{2}v)}\geq f(x),x\in D,$$
$$ \dfrac{S_{k}(D^{2}u)}{S_{l}(D^{2}u)}\geq f(x),x\in \mathbb{R}^n.$$
Moreover,
$$
u\leq v,x\in \o{D},
$$
$$u=v,x\in \partial D.$$
Set
$$w(x)=\left\{
\begin{array}{lll}
\vspace{2mm}
v(x),&x\in D,\\
\vspace{2mm} u(x),&x\in \mathbb{R}^n\backslash D.
\end{array}
\right.
$$
Then $w\in C^{0}(\mathbb{R}^n)$ is a $k-$convex function and
satisfies in the viscosity sense
$$ \dfrac{S_{k}(D^{2}w)}{S_{l}(D^{2}w)}\geq f(x),x\in \mathbb{R}^n.$$
\end{lemma}

\begin{proof}
See \cite{D}.
\end{proof}

\noindent{\bf{\large Acknowledgements}}

Dai is supported by Shandong Provincial Natural Science Foundation (ZR2021MA054). Bao is supported by NSF of China (11871102). Wang is supported by NSF of China (11701027) and Beijing Institute of Technology Research Fund Program for Young Scholars.

(L.M. Dai)  School of Mathematics and Information Science, Weifang
 University, Weifang, 261061, P. R. China

 {\it Email address}: lmdai@wfu.edu.cn
 \vspace{3mm}

(J.G. Bao)School of Mathematical Sciences, Beijing Normal University,
Laboratory of Mathematics and Complex Systems, Ministry of
Education, Beijing, 100875, P. R. China

{\it Email address}: jgbao@bnu.edu.cn
\vspace{3mm}

(B. Wang)School of Mathematics and Statistics, Beijing Institute of Technology, Beijing, 100081, P. R. China

{\it Email address}: wangbo89630@bit.edu.cn

\end{document}